\documentclass[10pt]{article}
\usepackage{amsfonts}
\usepackage{setspace}
\usepackage{pgf,tikz}
\usepackage{fullpage}
\usepackage{anysize}
\usepackage{float}
\usepackage{url}
\usepackage{hyperref}
\usepackage{geometry}
\usepackage{amssymb}
\usepackage{latexsym}
\usepackage{amsmath}
\usepackage{mathrsfs}
\usepackage{amsthm, amssymb}
\theoremstyle{definition}
\newtheorem{exmp}{Example}[section]
\usepackage{abstract}
\usepackage{amsmath}
\usepackage{lmodern}

\usepackage[english]{babel}
\usepackage{graphicx}

\def\XXint#1#2#3{{\setbox0=\hbox{$#1{#2#3}{\int}$ }
\vcenter{\hbox{$#2#3$ }}\kern-.6\wd0}}

\newtheorem{theorem}{Theorem}[section]
\newtheorem{definition}[theorem]{Definition}
\newtheorem{lemma}[theorem]{Lemma}
\newtheorem{proposition}[theorem]{Proposition}
\newtheorem{corollary}[theorem]{Corollary}
\makeatletter \renewenvironment{proof}[1][\proofname] {\par\pushQED{\qed}\normalfont\topsep6\p@\@plus6\p@\relax\trivlist\item[\hskip\labelsep\bfseries#1\@addpunct{.}]\ignorespaces}{\popQED\endtrivlist\@endpefalse} \makeatother

\newcommand{\diam}{\mbox{\rm diam}}              
\newcommand{\dist}{\mbox{\rm dist}}              
                
\newcommand{\loc}{\mbox{\rm loc}}
 
\newcommand{\convn}{\;\overrightarrow{_{_{n\rightarrow \infty}}}\;}

\newcommand{\convk}{\;\overrightarrow{_{_{k\rightarrow \infty}}}\;}

\mathchardef\mhyphen="2D

\begin{document}
\title{A Localized Diffusive Time Exponent for Compact Metric Spaces}
\author{John Dever \\
School of Mathematics\\
Georgia Institute of Technology\\
Atlanta GA 30332-0160}

\maketitle
\begin{abstract} We provide a definition of a new critical exponent $\beta$ that has the interpretation of a type of local walk dimension, and may be defined on any compact metric space. We then specialize to the case of random walks that jump uniformly in metric balls with respect to a given Borel measure of full support. We use the local exponent $\beta$ as a local time scaling exponent to re-normalize the time scale and produce approximating continuous time walks. We show a Faber-Krahn type inequality $\lambda_{1,r}(B)\geq \frac{c}{R^{\beta(x_0)}},$ where $c$ is a constant independent of $r$ and $x_0$ and where $\lambda_{1,r}(B)$ is the bottom of the spectrum of the generator for the re-normalized continuous time walk at stage $r$ killed outside of $B=B_R(x_0).$ In addition, we examine the local Hausdorff dimension $\alpha.$ We show that any variable Ahlfors $Q\mhyphen$regular measure is strongly equivalent to the local Hausdorff measure and that $Q=\alpha.$ We also provide new examples of variable dimensional spaces, including a variable dimensional Sierpinski carpet.
\end{abstract}


\maketitle

\section{Introduction}
It has become clear that the domain of the generator of a diffusion process on many non-homogeneous metric spaces such as fractals is often a type of Besov-Lipschitz function space characterized by an exponent $\beta.$ Unlike the case of Euclidean space, it often happens that Lipschitz functions are not in the domain of the generator, or Laplacian. Informally, to define the quadratic form of the Laplacian, instead of integrating the square of a gradient, one must integrate the square of a ``fractal gradient" of the form $``\frac{df}{dx^{\beta/2}}"$ for some exponent $\beta.$ From heat kernel asymptotics in many notable examples, it has been found that the heat kernel bounds are characterized by two scaling exponents, $\alpha$ and $\beta.$ The exponent $\alpha$ is the Hausdorff dimension, often appearing as a space scaling exponent in a suitable geometric measure. From the heat kernel bounds, one often finds that the expected square of the metric distance traveled by the process in a time $t$ scales like $t^{\frac{2}{\beta}}.$ Hence one gets the interpretation of $\beta$ as a kind of walk dimension. In many cases of interest this $\beta,$ in its guise as a walk dimension, is precisely the same $\beta$ that one should use in the ``fractal gradient" ``$\frac{df}{dx^{\beta/2}}$" in order to define the Laplacian. The problem, however, is how might one define this exponent $\beta$, preferably in a primarily geometric manner, without first knowing about functions in the domain of a possibly existing diffusion process. 

\subsection{Overview and Results}
In this paper we propose for a compact metric space a primarily geometric definition of a walk dimension exponent $\beta,$ defined purely in terms of the metric. Moreover, we show that this exponent may be localized, and indeed, there are natural examples where $\beta$ takes on a continuum of values. Our definition of the exponent $\beta$ appears to be new. As we shall see, it may be informally interpreted as a localized ``walk packing dimension" or as a local time scaling exponent.

We also consider another local exponent $\alpha,$ the local Hausdorff dimension. The exponent $\alpha$ is the local Hausdorff dimension, which has been considered previously in \cite{Loc} and \cite{dever} as well as implicitly in \cite{Sob}. It may be informally interpreted as a space scaling exponent, especially when considered in relation to a variable Ahlfors regular measure. A measure is variable Ahlfors regular when it satisfies the geometric property that the measure of a ball of a given radius scales like the radius to some power $Q$ depending on the center of the ball. We show that we must have $Q=\alpha.$ Moreover, we prove a kind of uniqueness result for variable Ahlfors regular measures, showing that any such measure is essentially equivalent, in a precise sense, to a local Hausdorff measure. Additionally, we provide several new examples of spaces in which $\alpha$ varies continuously, including a variable dimensional Sierpinski carpet.

An overview of the definition of $\beta$ is as follows. Given a compact metric space $X$ and a positive scale $\epsilon$, one may discretely approximate $X$ by a maximally separated set at that scale. Given such a discrete approximation, one may define in a geometric manner, a graph whose vertices are the elements of the approximating set. This then induces a discrete time random walk on the graph defined by jumping uniformly to an adjacent vertex. Given an ball $B$ of radius $R$ in $X$ and a vertex $y$, one may consider the expected time it takes for a 
walker on the approximating graph at scale $\epsilon$ starting at $y$ to 
leave the ball. One then defines $\beta(B)$ as a critical exponent where the behavior of the maximum exit time from the ball at scale $\epsilon$ multiplied by $\epsilon^\gamma$ changes, when $\gamma$ varies after letting $\epsilon\rightarrow 0$. One may show that if $B\subset B'$ where $B'$ is a ball, then $\beta(B)\leq \beta(B').$ We then define 
$\beta(x)$ as the infimum of the $\beta(B)$ where $B$ is a ball about $x.$ A precise definition is given in Section 4. We also argue there for our interpretation of $\beta$ as a kind of local walk packing dimension.

However, in this paper, for convenience, we mainly consider a $\beta$ defined with respect to $\epsilon \mhyphen$jump random walks on $X$ with respect to a given Borel measure $\mu$ of full support. That is, at stage $\epsilon$, given $x\in X,$ we assume that the walker jumps $\mu$-uniformly at $x$ in a ball of radius $\epsilon.$ Similarly we may consider the expected number of steps, $E_{\epsilon, B}(x)$, needed for a walker at stage $\epsilon$ to leave a ball $B$ starting at $x.$ One may then again define $\beta(B)$ as a critical exponent where $\sup_{y\in B}E_{\epsilon,B}(y)\epsilon^\gamma$ changes behavior in $\gamma$ as $\epsilon\rightarrow 0$ and $\beta(x)$ as a limit of $\beta(B_r(x))$ as $r\rightarrow 0.$

Once we have constructed $\beta,$ we use it to re-normalize the time scale of the discrete time walks by requiring that a walker at stage $\epsilon$ at site $x$ wait on average $\epsilon^{\beta(x)}$ before jumping to a neighboring site. This induces a continuous time walk $(X^{(\epsilon)}_t)_{t\geq 0}$. Given a ball $B$ we examine the expected exit time of the continuous time walk from $B.$ We will especially be interested in studying the case where the maximum expected exit time from a ball $B_r(x)$ scales like $r^{\beta(x)}$. Let us call such a condition $E_\beta$ (see also \cite{grigortelcs}). Under $E_\beta$, the exponent $\beta$ has the interpretation of a local time scale exponent.

Our primary motivation for the definition of $\beta$ is to attempt to define a suitable notion of a Laplace-Beltrami operator on $X$. As such, we consider the generator $\mathscr{L}_r$ of the continuous time walk $(X^{(r)}_t)_{t\geq 0}$ at stage $r.$ We then have that for $f\in L^2(X,\mu),$ 
\[\mathscr{L}_{r}f(x)=\frac{1}{r^{\beta(x)}\mu(B_r(x))}\int_{B_r(x)} (f(y)-f(x))d\mu(x).\]

It is known that on $\mathbb{R}^n$ (and indeed on any Riemannian $n-$manifold) that if $\mu$ is the standard Lebesgue measure (or volume measure on a Riemannian manifold) and if $f$ is any smooth function, then $\frac{1}{r^2\mu(B_r(x))}\int_{B_r(x)}(f(y)-f(x))d\mu(y)$ converges as $r$ approaches $0$ to a constant multiple of the Laplace-Beltrami operator evaluated at $f$ (See \cite{buragol}). Hence we should expect that for $x\in R^n,$ under its Euclidean metric with Lebesgue measure, that $\beta(x)=2.$ We show this to be the case. 

Moreover, by a variant of the well known Faber-Krahn inequality, if $B=B_R(0)$ is the ball of radius $R$ about the origin in $\mathbb{R}^n$ and $\lambda_1(B)$ is the first positive eigenvalue of (minus) the Dirichlet Laplacian on $B$, then 
\[\lambda_1(B)\geq CR^{-2},\] where $C$ is a constant independent of $R$. 

As a preliminary step towards form convergence estimates, we will establish the following as an analog of the Faber-Krahn inequality. Under the time scaling condition $E_\beta,$ if $\lambda_{1,r}(B)$ is the bottom of the spectrum of $\mathscr{L}_r$ with Dirichlet boundary conditions on a ball $B=B_R(x)$, then \[\lambda_{1,r}(B) \geq CR^{-\beta(x)},\] where $C$ is independent of $r, R,$ and $x.$

Additionally, under the assumption that the measure $\mu$ is variable Ahlfors regular, we show that $\beta(x)\geq 2$ for all $x\in X.$

Before continuing we note that our ultimate motivation for embarking on this line of research is to understand the conditions needed on a metric space to construct a strongly local, regular Dirichlet form. It is our hope that the results and ideas in this paper may be even a small step toward this goal.

We adopt the following notation: throughout the paper $f\asymp g$ means there exists a constant $C>0$, independent of the arguments of $f$, such that $\frac{1}{C}f\leq g\leq Cf$. 
By $\mathbb{Z}_+$ and $\mathbb{R}_+$ we mean the set of non-negative integers and the set of non-negative real numbers, respectively.

\subsection{Related Work}

Similar exit time scaling exponents and power law scaling conditions on such exponents in various contexts have been considered by a myriad of other authors. 

There is a notion of a walk dimension found in the literature on fractal graphs. In \cite{telcs1}, a local exponent $d_W(x)$ is defined for a random walk on an infinite graph $G$ as follows. If $x$ is a vertex and $N>0$ an integer, let $E_N(x)$ be the expected number of steps needed for a random walk on $G$ starting at $x$ to reach of vertex of graph distance more than $N$ away from $x.$ In other words $E_N(x)$ is the expected exit time of the walk from the ``graph ball" of graph distance (or ``chemical distance") $N$ about $x.$ Then set $d_W(x):=\limsup_{N\rightarrow \infty} \frac{\log(E_N(x))}{\log(N)}.$ In the literature on random walks on infinite graphs, scaling 
conditions of the exit time $E_R(x)$ from a graph ball of graph radius 
$R$ about $x$ of the form $E_R(x)\asymp R^\beta$ have been considered  
\cite{telcs2}, \cite{grigortelcs}, \cite{barlowesc}. It is clear that if a graph satisfies 
such a condition then $\beta$ must be $d_W$ as defined above.  Barlow has shown in \cite{barlowesc} that if a graph satisfies such an exit time condition and an additional volume scaling condition analogous to Ahlfors regularity, then it 
must be the case that $2\leq d_W \leq 1+\alpha,$ where $\alpha$ is a dimension arising from a volume scaling condition of graph balls analogous to Ahlfors regularity. In \cite{telcseins} conditions are given for the so called Einstein relation to hold connecting resistance growth and volume growth on annuli to mean exit time growth on graph metric balls. The monograph \cite{art} presents an excellent exposition and overview of the general theory of random walks on infinite graphs. Moreover, Theorems 7.7 and 7.8 presented here follow in part ideas presented in the proofs for the graph case in Lemma 2.2 and 2.3 in \cite{art}.

On many fractals such as the Sierpinski gasket and carpet it is known that the fractal may be represented by an infinite graph. On the Sierpinski gasket one may compute explicitly the mean exit time from graph metric balls, and one finds that $\beta=\frac{\log5}{\log 2}.$ For the Sierpinski carpet it is known that it must satisfy an exit time scaling condition for some power of $\beta.$ However the exact value of $\beta$ in this case is unknown.

In the setting of metric measure Dirichlet spaces, a walk dimension $\beta$ has also appeared in certain sub-diffusive heat kernel estimates on various fractals and infinite fractal graphs. For a wide class of fractals, including the Sierpinski gasket and carpet, for which diffusion processes are known to exist, heat kernel estimates of the form 
\[p_t(x,y)\asymp c_1t^{-\frac{\alpha}{\beta}}\exp\left({-c_2\left(\frac{d(x,y)^{\beta}}{t}\right)^{\frac{1}{\beta-1}}}\right)\]
have been shown to hold \cite{barlowgasket}, \cite{barlowcarpet}, \cite{fitz}, \cite{grigor2012two}. It is known that such heat kernel estimates imply that the underlying mesaure is Ahlfors regular with Hausdorff dimension $\alpha$ \cite{grigorlau}, thus providing us with another motiviation for its study. The exponent $\beta$ appearing in these estimates is called the walk dimension \cite{barlowcarpet1}. Moreover, in the setting of metric measure Dirichlet spaces, consequences of a mean exit time scaling condition similar to $E_\beta$ is considered in \cite{grigor2012two}. Such conditions, together with volume doubling and an elliptic Harnack inequality, have been shown to imply the existence of a heat kernel along with certain heat kernel estimates \cite{grigor2012two}. Additionally, an adjusted Poincar\'{e} inequality involving the mean exit time has been proposed in \cite{bass2013} and \cite{barlowcarpet1}. Such Poincar\'{e} inequalities or resistance estimates together with an elliptic Harnack inequality often play an important role in the proofs of the existence of diffusion processes (See \cite{kusuokazhou}, \cite{barlowres}.) 

In the setting of Riemannian geometry, the function giving the mean exit time from a ball starting at a given point is known as the torsion function, and the integral of the mean exit time function is known as the torsional rigidity \cite{vand}. 

It is known that the domains of many diffusions on fractals are a type of Besov-Lipschitz function space \cite{jon}, \cite{kumagai}. For $\sigma>0, r>0,$ let \[\mathscr{E}_{r,\sigma}(f)=\int_X\frac{1}{r^\sigma\mu(B_r(x))}\int_{B_r(x)}(f(y)-f(x))^2d\mu(y)d\mu(x).\] Then let \[W^\sigma(X,\mu):=\{f\in L^2(X,\mu)\;|\;\sup_{r>0}\mathscr{E}_{r,\sigma}(f)<\infty\}.\] Then $W^\sigma(X,\mu)$ is a Banach space with norm $\|f\|^2_{W^\sigma(X,\mu)}=\|f\|^2_{L^2(X,\mu)}+ \sup_{r>0}\mathscr{E}_r(f).$ A potential theoretic definition of a walk dimension has been given as a critical exponent $\beta^*$, obtained by varying $\sigma$, where $W^\sigma(X,\mu)$ changes behavior to containing only constant functions \cite{grigorlau}, \cite{sturm}. Moreover, in \cite{grigorlau}, conditions were given for $\beta^*$ to equal $\beta,$ provided $\beta$ may be defined from heat kernel estimates. In \cite{gu} it was proven that $\beta^*$ is a Lipschitz invariant among metric measure spaces with an Ahlfors regular measure.

Recently, a proposal for a method to define $\beta^*$ without reference to diffusion was proposed by Grigor'yan \cite{grigorwalk}. The method, applied there to the Sierpinski gasket, involves the procedure of forming a weighted hyperbolic graph induced from the graph approximations to the space and seeing the original space as a Gromov hyperbolic boundary (See, for instance,  \cite{lau} and \cite{Piaggio} for more on this method). A random walk on the hyperbolic graph induces a non-local form on the boundary whose domain is another type of Besov-Lipschitz space. Again, $\beta^*$ is seen as a critical exponent where the space changes to have sufficiently many non-constant functions. The hyperbolic graph approximation allows one to examine these functions in terms of the random walk on the hyperbolic graph.
 
The local Hausdorff dimension was defined in \cite{Loc}. A curve with continuously varying local dimension was considered, somewhat informally, in \cite{nottale}. A variable dimensional Koch curve and a local Hausdorff measure was defined in \cite{Sob}. Also, variable Ahlfors $Q(\cdot)\mhyphen$regular measures were considered in \cite{Sob}.  For $Q=d$ constant, it is known that $d=\dim_H(X)$ and that if $\mu$ is any other Ahlfors $d\mhyphen$regular measure, then $H^d\asymp \mu,$ where $H^d$ is the Hausdorff measure at dimension $d$ \cite{Hein}. 

In \cite{burago2013graph} and \cite{buragol} the authors used a discrete approximation with weighted $\epsilon\mhyphen$net graphs and a continuous approximation with a given Borel measure of full support, respectively, to create approximate Dirichlet forms on a compact metric space. In \cite{sturm} variational ($\Gamma\mhyphen$)convergence was used to study limits of certain approximating forms defined both through approximating graphs and by means of a given Borel measure of full support. Additionally, in \cite{sturm}, sufficient conditions were given for a ($\Gamma\mhyphen$)limit of such approximating forms to generate a non-trivial diffusion process.

\subsection{Organization}
This paper is organized as follows:

In Section 2, we define the local space scaling exponent $\alpha$ and, following \cite{Sob}, a local Hausdorff measure. 

We then, in Section 3, again following \cite{Sob}, define variable Ahlfors regularity. We prove that the variable exponent of a variable Ahlfors regular measure must be $\alpha.$ Moreover, we prove that, up to what we call strong equivalence of measures, the local Hausdorff measure is the only possible Ahlfors regular measure. 

In Section 4 we review the example of the variable dimensional Koch curve found in \cite{Sob}. Moreover, we propose similarly constructed examples of a variable dimensional Sierpinski gasket, a variable dimensional Sierpinski carpet, and a variable dimensional Vicsek tree. 

In Section 5 we define the local exponent $\beta$ in the case of approximation by $\epsilon$-nets. We discuss its interpretation as a local walk packing dimension and show that in the case of the variable dimension

In Section 6, we define $\beta$ in the case of approximate continuous random walks defined via a given Borel measure $\mu$ of full support.  Then, using this definition, we use $\beta$ to define local waiting times, allowing us to ``re-normalize" the time scale according to an assumption of power law scaling of the local time scale with respect to the space scale. 

Lastly, in Section 7, we examine properties of the approximate continuous time random walks. In particular, under the scaling assumption on the re-normalized exit time ($E_\beta$) we obtain a Green function for the generator and a Faber-Krahn type spectral inequality for the Dirichlet generator. Then, under the assumption that $\mu$ is variable Ahlfors regular, we prove that $\beta\geq 2.$

We hope to address the convergence of approximating Dirichlet forms in future research.

\section{Local Hausdorff Dimension and Measure}

For $(X,d)$ a metric space, we denote the open ball of radius $0\leq r\leq \infty$ about $x\in X$ by $B_r(x)$. We denote the closed ball of radius $r$ by $B_r[x].$
For $A\subset X,$ we let $|A|:=\diam(A)=\sup_{[0,\infty]}\{d(x,y)\;|\;x,y\in A\}.$ Throughout the paper, let $\mathscr{C}:=\mathscr{P}(X)$ be the power set of $X,$ and let $\mathscr{B}$ be the collection of all open balls in $X,$ where $\emptyset = B_0(x)$ and $X=B_\infty(x)$ for 
any $x\in X.$ By a covering class we mean a collection $\mathscr{A}\subset \mathscr{P}(X)$ with $\emptyset, X \in \mathscr{A}.$ We will primarily work with the covering classes $\mathscr{C}=\mathscr{P}(X)$ and $\mathscr{B}.$ For $\mathscr{A}$ a covering class and $A\subset X$, let $\mathscr{A}_\delta(A) := \{\mathscr{U} \subset \mathscr{A} \;|\;$\mbox{$\mathscr{U}$ at 
most countable},\;$ |U|\leq \delta$ \mbox{for}\;$ U\in \mathscr{U},\; A\subset \cup 
\mathscr{U}\}.$ 

Recall that an outer measure on a set $X$ is a function $\mu^*:\mathscr{P}(X)\rightarrow [0,\infty]$ such that 
$\mu^*(\emptyset)=0$; if $A,B\subset X$ with $A\subset B$ then $\mu^*(A)\leq \mu^*(B);$ and if 
$(A_i)_{i=1}^\infty \subset \mathscr{P}(X)$ then 
$\mu^*(\cup_{i=1}^\infty) A_i)\leq \sum_{i=1}^\infty 
\mu^*(A_i).$ The second condition is called monotonicity, and the last condition is called countable subadditivity.

A measure $\mu$ on a $\sigma\mhyphen$algebra $\mathscr{M}$ is called complete if for all $N\in \mathscr{M}$ with $\mu(N)=0,$ $\mathscr{P}(N)\subset \mathscr{M}$.

The following theorem due to Carath\'{e}odory is basic to the subject. See \cite{Folland} for a proof.
\begin{theorem}  If $\mu^*$ is an outer measure on $X$ then if $\mathscr{M^*}:=\{A\subset X\;|\;\forall E\subset X\;[\;\mu^*(E)=\mu^*(E\cap A)+\mu^*(E\cap A^c)\;]\;\},$ $\mathscr{M^*}$ is a $\sigma\mhyphen$algebra and $\mu^*|_{\mathscr{M^*}}$ is a complete measure. 
\end{theorem}

Now suppose $(X,d)$ is a metric space. Sets $A,B\subset X$ are called positively separated if $\dist(A,B)=\inf\{d(x,y)\;|\;x\in A,y\in B\;\}>0.$ An outer measure $\mu^*$ on $X$ is called a metric outer measure if for all $A,B\subset X$ with $A,B$ positively separated, $\mu^*(A\cup B)=\mu^*(A)+\mu^*(B).$ 

The Borel sigma algebra is the smallest $\sigma\mhyphen$algebra containing the open sets of $X.$ Elements of the Borel $\sigma\mhyphen$algebra are called Borel sets. A Borel measure is a measure defined on the $\sigma\mhyphen$algebra of Borel sets. The following proposition is well known. See \cite{Falc1} for a proof.

\begin{proposition} \label{prop2.2}
If $\mu^*$ is a metric outer measure on a metric space $X,$ then $\mathscr{M}^*$ contains the $\sigma\mhyphen$algebra of Borel sets. In particular, $\mu^*$ may be restricted to a Borel measure. 
\end{proposition}

For $\tau:\mathscr{C}\rightarrow [0,\infty]$ with $\tau(\emptyset)=0,$
let $\mu^* _{\tau, \delta} (A):=\inf \{ \sum_{U\in \mathscr{U}}\tau(U)\;|\; 
\mathscr{U}\in \mathscr{C}_\delta(A)\}$ and $\mu^* _{\tau}(A)=\sup_{\delta>0} \mu^* _{\tau, \delta} 
(A).$ 
The verification of the following proposition is straightforward. 
\begin{proposition}  $\mu^* _{\tau}$ is a metric outer measure. \end{proposition}

Note we may restrict to any covering class $\mathscr{A}$ containing $\emptyset$ and $X$ by setting $\tau(U)=\infty$ for $U\in \mathscr{C}\setminus \mathscr{A}.$

It then follows by Proposition \ref{prop2.2} that the $\mu^* _{\tau}$ measurable sets contain the Borel $\sigma\mhyphen$algebra. Let $\mu_\tau$ be the restriction of $\mu^* _{\tau}$ to the Borel sigma algebra. Then by Carath\'eodory's Theorem, $\mu_\tau$ is a Borel measure on $X.$ 

For $s\geq 0$ let $H^s$ be the measure obtained from the choice $\tau(U)=|U|^s$ for $U\neq \emptyset$ and $\tau(\emptyset)=0.$ If $U$ is non-empty and $|U|=0$ we adopt the convention $|U|^0=1$. $H^s$ is called the $s$-dimensional Hausdorff measure. Let $\lambda^s$ be the measure obtained by restricting $\tau$ to the smaller covering class $\mathscr{B}$ of open balls. Concretely, let $\lambda^s$ be the measure obtained by setting $\tau(B)=|B|^s$ for $B$ a non-empty open ball, $\tau(\emptyset)=0,$ and $\tau(U)=\infty$ otherwise. We call $\lambda^s$ the $s$-dimensional open spherical measure. 

We call Borel measures $\mu,\nu$ on $X$ \textit{strongly equivalent}, written $\mu \asymp \nu$, if there exists a constant $C>0$ such that for every Borel set $E$, $\frac{1}{C}\nu(E)\leq \mu(E)\leq C\nu(E).$ 

The proof of the following lemma, while straightforward, is included for completeness.
\begin{lemma} For any $s\geq 0,$ $\lambda^s \asymp H^s.$ 
\end{lemma}
\begin{proof} It is clear that $H^s \leq \lambda^s.$ Conversely, let $A$ be a Borel set. We may assume $A\neq \emptyset.$ Also, we may assume $H^s(A)<\infty,$ since 
otherwise the reverse inequality is clear. If $s=0$ then $H^0$ is the counting measure. So let $H^0(A)=n<\infty.$ Let $x_1,...,x_n$ be an enumeration of the elements of $A.$ Let $r$ be the minimum distance between distinct elements of $A.$ For $0<\delta<r$ let $B_i = 
B_{\frac{\delta}{2}}(x_i)$ for each $i.$ Then $(B_i)_{i=1}^n \in \mathscr{B}_\delta(A)$ 
and $\lambda^{0,*}_\delta(A)\leq n=H^0(A)$. Hence $\lambda^0(A) \leq H^0(A).$ So we may 
assume $s>0.$  Let $\epsilon>0.$ Let $\delta>0$ and let $\mathscr{U}\in 
\mathscr{C}_\delta(A)$ with $\sum_{U\in \mathscr{U}} |U|^s <\infty.$ We may assume $U\neq 
\emptyset$ for each $U\in \mathscr{U}.$ Choose $x_U\in U$ for each $U\in \mathscr{U}.$ 
Let $\mathscr{U}_0=\{U\in \mathscr{U}\;|\;|U|=0\}, \mathscr{U}_1:=\{U\in \mathscr{U}\;|
\;|U|>0\}.$ Then for each $U\in \mathscr{U}_0$ choose $0<r_U<\delta$ such that 
$\sum_{U\in \mathscr{U}_0} r_U^s<\frac{\epsilon}{2^s}.$ For $U\in \mathscr{U}_1$ let $r_U = 2|U|.$ Then let $B_U:=B_{r_U}(x_U)$ for $U\in \mathscr{U}.$ It follows that 
$(B_U)_{U\in \mathscr{U}}\in \mathscr{B}_{4\delta}(A)$ and $\lambda^{s,*}_{4\delta}
(A)\leq \sum_{U\in \mathscr{U}}|B_U|^s \leq 2^s(\sum_{U\in \mathscr{U}_0} r_U^s + 
\sum_{U\in \mathscr{U}_1} r_U^s) \leq \epsilon + 4^s\sum_{U\in \mathscr{U}} |U|^s.$ Hence 
$\lambda^s(A)\leq 4^sH^s(A).$ \end{proof}

Let $X$ be a metric space and $A\subset X$ with $A$ non-empty.  Let $0\leq t<s.$ If $(U_i)_{i=1}^\infty \in \mathscr{C}_\delta(A)$ then $H^s_\delta(A)\leq \sum_i |U_i|^s \leq \delta^{s-t}\sum_i |U_i|^t.$ So $H^s_\delta(A) \leq \delta^{s-t} H_\delta^t(A)$ for all $\delta>0$.

Suppose $H^t(A)<\infty$. Then since $\delta^{s-t}\overrightarrow{_{_{\delta \rightarrow 0^+}}} 0,$ $H^s(A)=0.$ Similarly, if $H^s(A)>0$ and $t<s,$ $H^t(A)=\infty.$ It follows \[\sup\{s\geq 0\;|\;H^s(A)=\infty\}=\inf\{s\geq 0\;|\;H^s(A)=0\}.\] We denote the common number in $[0,\infty]$ by $\dim(A).$ It is called the Hausdorff dimension of $A.$ Since $H^s\asymp \lambda^s,$ we also have \[\dim(A)= \sup\{s\geq 0\;|\;\lambda^s(A)=\infty\}=\inf\{s\geq 0\;|\;\lambda^s(A)=0\}.\]
We adopt the convention $\dim(\emptyset)=-\infty.$

\begin{lemma} \label{mon} If $A\subset B \subset X$ then $\dim(A)\leq \dim(B).$\end{lemma}
\begin{proof} By monotonicity of measure, $H^s(A)\leq H^s(B)$. So $H^s(B)=0$ implies 
$H^s(A)=0.$ Therefore $\dim(A)=\inf\{s\geq 0\;|\;H^s(A)=0\}\leq \inf\{s\geq 0\;|
\;H^s(B)=0\}=\dim(B).$\end{proof} 

Let $\mathscr{O}(X)$ be the collection of open subsets of $X.$ For $x\in X,$ let $\mathscr{N}(x)$ be the open neighborhoods of $x.$ Define the \textit{local dimension} $\alpha:X\rightarrow [0,\infty]$ by \[\alpha(x):=\inf \{\dim(U)\;|\; U\in 
\mathscr{N}(x)\}.\]   Also by \ref{mon}, \[\alpha(x)=\inf\{\dim(B_\epsilon(x))\; | \;\epsilon>0\}.\] 

\begin{proposition} The local dimension $\alpha$ is upper semicontinuous. In particular, it is Borel measurable and bounded above. \label{semicont}
\end{proposition}
\begin{proof} Let $c\geq 0.$ If $c=0$ then $\alpha^{-1}([0,c))=\emptyset\in\mathscr{O}
(X)$. So let $c>0.$ Then suppose $x \in \alpha^{-1}([0,c)).$ Then there exists a $U\in 
\mathscr{N}(x)$ such that $\dim(U)<c.$ Then for $y\in U$, since $U$ is open there exists a $V\in \mathscr{N}(y)$ with $V\subset U.$ So $\alpha(y)\leq \dim(V)\leq \dim(U)<c.$ Hence $x\in U \subset \alpha^{-1}([0,c)).$ Therefore $\alpha^{-1}([0,c))$ is open. Note the sets $\alpha^{-1}([0,n))$ for $n\in \mathbb{Z}_+$ form an open cover of $X$. Since $X$ is compact there exists an $N\in \mathbb{Z}_+$ such that $X=\alpha^{-1}([0,N)).$ Hence $\alpha$ is bounded above.
 
\end{proof}

\begin{lemma} If $A$ is a Borel set and $0\leq s_1\leq s_2$ then $\lambda^{s_2}(A)\leq \lambda^{s_1}(A).$ 
\end{lemma}
\begin{proof} If $\mathscr{U}\in \mathscr{B}_\delta(A)$ then $|U|^{s_2}\leq |U|^{s_1}$ for $0<\delta<1.$ The result then follows from the definition of $\lambda^s.$
\end{proof}

For $U\subset X, U\neq \emptyset,$ let $\tau(U)=|U|^{\dim(U)}.$ Set $\tau(\emptyset)=0.$ Then $H_{\loc}:=\mu_\tau$ is called the \textit{local Hausdorff measure}. If we restrict $\tau$ to $\mathscr{B}$ then the measure $\lambda_{\loc}:=\mu_\tau$ is called the \textit{local open spherical measure}. 

\begin{lemma} \label{asym} If $\dim(X)<\infty$ then $H_{\loc} \asymp \lambda_{\loc}.$\end{lemma}
\begin{proof} Clearly $H_{\loc} \leq \lambda_{\loc}.$ Let $A$ be a Borel set. We may assume $H_{\loc}(A)<\infty$ and $A\neq \emptyset.$ Let $\epsilon>0,0<\delta<\frac{1}{4},$ $\mathscr{U}\in \mathscr{C}_\delta(A)$ with $U\neq \emptyset$ for $U\in \mathscr{U}$ and 
$\sum_{U\in \mathscr{U}}|U|^{\dim(U)}<\infty.$ Let $x_U\in U$ for each $U\in \mathscr{U}.
$ If $|U|=0$ then $U=\{x_U\}$ is a singleton and so $\dim(U)=0.$ Then, by our convention, $|U|^{\dim(U)}=1.$ Hence there are at most finitely 
many $U\in \mathscr{U}$ with $|U|=0.$ Let $\mathscr{U}_0$ be the collection of such 
$U\in\mathscr{U}.$ Let $m:=\min_{U\in \mathscr{U}_0}\alpha(x_U).$ let $0<r<\delta/2$ 
such that $(2r)^m\leq 1.$ Then, for $U\in \mathscr{U}_0,$ $U\subset B_r(x_U)$ and, since 
$0<2r<\delta<1$ and $m\leq \alpha(x_U)\leq \dim(B_r(x_U)),$ $|B_r(x_U)|
^{\dim(B_r(x_U))}\leq (2r)^m \leq 1.$ For $U\in \mathscr{U}_0$ set $r_U:=r$. Let 
$\mathscr{U}_1$ be the collection of $U\in \mathscr{U}$ with $|U|>0.$ For $U\in 
\mathscr{U}_1$ let $r_U:=2|U|.$ Then for $U\in \mathscr{U}$ let $B_U:=B_{r_U}(x_U).$ Then 
$U\subset B_U$, $(B_U)_{U\in \mathscr{U}} \in \mathscr{B}_{4\delta}(A),$ and, since $|B_U|\leq 4|U|\leq 4\delta<1$ 
and $\dim(U)\leq \dim(B_U) \leq \dim(X)<\infty$ for $U \in \mathscr{U},$ $\sum_{U\in 
\mathscr{U}} |B_U|^{\dim(B_U)} \leq \sum_{U\in \mathscr{U}_0} 1 + 
\sum_{U \in \mathscr{U}_1} (4|U|)^{\dim(U)} \leq 4^{\dim(X)} \sum_{U \in \mathscr{U}}|U|^{\dim(U)}.$ 
Hence $\lambda_{\loc} \leq 4^{\dim(X)}H_{\loc}.$
\end{proof}

\begin{proposition}
If $d_0$ is the dimension of $X,$ then $H^{d_0}\ll  H_{\loc}.$ 
\end{proposition}
\begin{proof} 
Let $\phi(U)=|U|^{d_0}$, $\tau(U)=|U|^{\dim(U)}$ where $\phi(\emptyset)=0, \tau(\emptyset)=0.$ By definition, since $\dim(U)\leq d_0$ for all $U,$ if $0<\delta<1$ then for any set $A,$ 
\[\mu^*_{\phi,\delta}(A)\leq \mu^*_{\tau,\delta}(A).\] 

Suppose $N\subset X$ is Borel measurable with $H_{\loc}(N)=0.$ Let $\epsilon>0.$ Then for all $\delta>0$ there exists a $\mathscr{U}_\delta \in \mathscr{C}_\delta(N)$ such that $\sum_{U\in \mathscr{U}_\delta}|U|^{\dim(U)}<\epsilon.$ But since for  $U\in \mathscr{U}_\delta,$ $\{U\}\in \mathscr{C}_\delta(U),$ and since $\mu^*_{\phi,\delta}$ is an outer measure, for $0<\delta<1$ we have
\[\mu^*_{\phi,\delta}(N)\leq \sum_{U\in \mathscr{U}_\delta}\mu^*_{\phi,\delta}(U)\leq 
\sum_{U\in \mathscr{U}_\delta}\mu^*_{\tau,\delta}(U)\leq \sum_{U\in \mathscr{U}_\delta}|U|^{\dim(U)}<\epsilon.\] Hence $H^{d_0}(N)\leq \epsilon.$  Since $\epsilon>0$ was arbitrary, $H^{d_0}(N)=0.$
\end{proof}

The following two propositions relate the local dimension to the global dimension.

\begin{proposition} \label{3.6} Suppose $X$ is separable with Hausdorff dimension $d_0.$ Let $A:=\alpha^{-1}([0,d_0)).$ Then $H^{d_0}(A)=0.$ In particular, $H^{d_0}(X)=H^{d_0}(\alpha^{-1}(\{d_0\}))$.
\end{proposition}
\begin{proof}
$A$ is open since $\alpha$ is upper semicontinuous. We may assume $A$ is non-empty. For $x\in A$ let $U_x\in \mathscr{N}(x)$ with $\dim(U_x)<d_0$ and $U_x\subset A.$ Then the $U_x$ form an open cover of $A.$ Since $X$ has a countable basis, there exists a countable open cover $(U_k)$ of $A$ with the property that for all $x$ there exists a $k$ with $x\in U_k \subset U_x.$ In particular $\dim(U_k)\leq \dim(U_x)<d_0.$ Let $A_j:=\cup_{k\leq j} U_k.$ Then, since $H^s(A_j)\leq \sum_{k\leq j} H^s(U_k)$ for any $s\geq 0,$ $\dim(A_j)\leq \max_{k\leq j} \dim(U_k)<d_0.$  So $H^{d_0}(A_j)=0$ for all $j.$ But by continuity of measure, $H^{d_0}(A)=\sup_j H^{d_0}(A_j)=0.$ Since $\alpha\leq d_0$ the other result follow immediately. 
\end{proof}

\begin{proposition} \label{dimmax}
Let $X$ be a separable metric space. Then $\dim(X)=\sup_{x\in X}\alpha(x).$ Moreover, if $X$ is compact then the supremum is attained. 
\end{proposition}

\begin{proof}
Clearly $\sup_{x\in X}\alpha(x) \leq \dim(X).$ Conversely, let $\epsilon>0.$ For $x\in X$ let $U_x\in \mathscr{N}(x)$ such that $\dim(U_x)\leq \alpha(x)+\frac{\epsilon}{2}.$ Then the $(U_x)_{x\in X}$ form an open cover of $X.$ Since $X$ is separable it is Lindel\"{o}f. So let $(U_{x_i})_{i\in \mathbb{Z}_{+}}$ be a countable subcover. Then if $\sup_{i\in \mathbb{Z}_+}\dim(U_{x_i})=\infty$ then also $\dim(X)=\infty.$ Else if $\sup_{i\in \mathbb{Z}_+}\dim(U_{x_i})<t<\infty$ then $H^t(U_{x_i})=0$ for all $i$ and so $H^t(X)\leq \sum_{i=1}^\infty H^t(U_{x_i})=0.$ So $\dim(X)\leq t.$ Hence $\dim(X)=\sup_{i\in \mathbb{Z}_+}\dim(U_{x_i}).$ Then choose $j\in \mathbb{Z}_+$ such that $\dim(X)\leq \dim(U_{x_j})+\frac{\epsilon}{2}.$ Then $\dim(X)\leq \alpha(x_j)+\epsilon\leq \sup_{x\in X}\alpha(x)+\epsilon.$ Hence $\dim(X)=\sup_{x\in X}\alpha(x).$

Now suppose $X$ is compact. Let $d_0:=\dim(X).$ It remains to show that there exists some $x\in X$ with $\alpha(x)=d_0.$ Suppose not. Then clearly $d_0>0.$ Let $m\geq 1$ be an integer such that $\frac{1}{m}<d_0.$ Then the sets $U_n:=\alpha^{-1}[0,d_0-\frac{1}{n})$ for $n\geq m$ form an open cover of $X.$ By compactness there exists a finite subcover. So there exists an $N>0$ such that $X=\alpha^{-1}[0,d_0-\frac{1}{N}).$ Hence $\sup_{x\in X} \alpha(x)\leq d_0-\frac{1}{N}<d_0,$ a contradiction.

\end{proof}

\section{Variable Ahlfors Regularity}

Recall that a metric space $(X,d)$ is Ahlfors regular of exponent $d_0\geq 0$ if there exists a 
Borel measure $\mu$ on $X$ and a $C>1 $ such that for all $x\in X$ and all $0<r\leq 
\diam(X),$ \[\frac{1}{C}r^{d_0}\leq \mu(B_r(x))\leq Cr^{d_0}.\] It is well known that if $X$ 
supports such a measure $\mu$  then $d_0$ is the Hausdorff dimension of $X$ and $\mu$ is 
strongly equivalent to the Hausdorff measure $H^{d_0}$ \cite{Hein}\cite{Fed}.

In this section we generalize this result on a compact metric space. If $Q:X\rightarrow (0,\infty)$ is a bounded function, then a measure $\nu$ is called \textit{(variable) Ahlfors $Q\mhyphen$regular} if there exists a constant $C>1$ so that \[\frac{1}{C} \nu(B_r(x))\leq  r^{Q(x)}\leq C\nu(B_r(x))\] for all $0<r\leq \diam(X)$ and $x\in X.$\cite{Sob} We show that if $X$ is compact and supports such a measure $\nu$ then $Q$ is the local Hausdorff dimension and $\nu$ is strongly equivalent to the local Hausdorff measure $H_{\loc}.$
Our presentation in this section is strongly influenced by \cite{Sob}.

Suppose $(X,d)$ is compact. For $Q:X\rightarrow [0,\infty)$ continuous, define $Q^-,Q^+, Q^c:\mathscr{B}\rightarrow [0,\infty)$ by $Q^-(U)=\inf_{x\in U} Q(x), Q^+(U)=\sup_{x\in U} Q(x).$ For arbitrary $Q:X\rightarrow [0,\infty)$ define $Q_c:\mathscr{B}\rightarrow 
[0,\infty)$ by $Q_c(B_r(x))=Q(x).$ Then for $\tilde{Q}:\mathscr{B}\rightarrow [0,\infty)$ with $Q^- \leq \tilde{Q}\leq Q^+,$ let $\lambda^{\tilde{Q}}:=\mu_\tau$, where $\tau$ is restricted to $\mathscr{B}$ and defined by $\tau(B):=|B|^{\tilde{Q}(B)}, \tau(\emptyset)=0.$

For $Q:X\rightarrow [0,\infty)$, we call $X$ \textit{$Q\mhyphen$amenable} if $0<\lambda^{Q_c}(B)<\infty$ for every non-empty open ball $B$ of finite radius in $X.$ In the case of $X$ compact this is equivalent to $\lambda^{Q_c}$ being finite with full support.

\begin{proposition} \label{amen} If $X$ is $Q\mhyphen$amenable with $Q$ continuous then $\alpha(x)=Q(x)$ for all $x\in X.$ 
\end{proposition}
\begin{proof} Let $B=B_r(x)$ be a non-empty open ball, $q^-:=Q^-(B), q^+:=Q^+(B), d_0:=\dim(B).$ Let $B':=B_{\frac{r}{2}}(x)$ Then if $d_0<q^-, \dim(B')\leq d_0< q^-$ so $\lambda^{q^-}(B')=0.$ Let $0<\delta<\min\{\frac{r}{8}, 1\}.$ Let $\mathscr{U}\in \mathscr{B}_\delta(B')$ and $U\in \mathscr{U}.$ We may assume $U\cap B' \neq \emptyset$, since otherwise $\mathscr{U}$ may be improved by removing such a $U.$ Say $U=B_{r_U}(x_U).$
Moreover, we may assumue $r_U\leq 2\delta.$ Indeed, if $|U|=0$ and $r_U>\delta$ then $U=B_{\delta}(x_U)$ so we may take $r_U=\delta$ in that case. If $|U|>0$ then if $r_U>2|U|$ then $B_{r_U}(x_U) = B_{2|U|}(x_U)$ so we may take $r_U = 2|U|\leq 2\delta.$ Say $w\in U\cap B'.$ Then if $z \in U,$ 
$d(z,x)\leq d(z,x_U)+d(x_U,w)+d(w,x)\leq 2r_U+\frac{r}{2} \leq 4\delta+\frac{r}{2}<r.$ So $U\subset B.$ Hence $q^- \leq Q^c(U)$ and since $|U|\leq 1,$ $|U|^{q^-}\geq |U|^{Q^c(U)}$. So $\lambda^{Q_c}(B')=0,$ a contradiction. If $d_0>q^+$ then since $B_r(x)=\cup_{n=1}^\infty B_{r-\frac{1}{n}}(x),$ it is straightforward, using countable subadditivity of the measures $H^s$ and the definition of Hausdorff dimension, to verify that $\dim(B_r(x)) = \sup_{n\geq 1} \dim(B_{r-\frac{1}{n}}(x))$. Since $d_0>q^+,$ let $N$ so that $\dim(B_{r-\frac{1}{N}}(x))>q^+.$ Let $r'=r-\frac{1}{N}$ and $B'=B_{r'}(x).$ Then $\lambda^{q^+}(B') = \infty.$ Let $0<\delta<\frac{1}{4N}.$ Let $\mathscr{U}\in \mathscr{B}_\delta(B')$ and $U\in \mathscr{U}.$ We may assume $U\cap B' \neq \emptyset,$ say $w\in U\cap B'$. As before we may assume $r_U\leq 2\delta$ and so if $z\in U$ then $\rho(z,x)\leq 2r_u+r'\leq 4\delta+r'<r.$ So $U\subset B.$ Hence $Q_c(U)\leq q^+$. Since $|U|<1,$ $|U|^{q^+}\leq |U|^{Q_c(U)}.$  Hence $\lambda^{Q_c}(B')=\infty,$ a contradiction. Hence $Q^-(B)\leq \dim(B) \leq Q^+(B)$ for every non-empty open ball $B.$ 
The result then follows since $Q$ is continuous.  
\end{proof}

A Borel measure $\nu$ on a metric space $X$ is said to have local dimension $d_\nu(x)$ at $x$ if $\lim_{r\rightarrow 0^+} \frac{\log(\nu(B_r(x)))}{\log(r)}=d_\nu(x)$. Since the limit may not exist, we may also consider upper and lower local dimensions at $x$ by replacing the limit with an upper or lower limit, respectively. 

It can be immediately observed that if $\nu$ is Ahlfors $Q\mhyphen$regular then $d_\nu(x)=Q(x)$ for all $x$.

A function $p$ on a metric space $(X,d)$ is \textit{log-H{\"o}lder continuous} if 
there exists a $C>0$ such that $|p(x)-p(y)|\leq \frac{-C}{\log(d(x,y))}$ for all $x,y$ with $0<d(x,y)<\frac{1}{2}.$

The following is may be found in \cite{Sob} (Proposition 3.1).

\begin{lemma} For $X$ compact, if $Q:X\rightarrow (0,\infty)$ log-H{\"o}lder continuous, $\tilde{Q}:\mathscr{B}\rightarrow [0,\infty)$ with $Q^-\leq \tilde{Q} \leq Q^+,$ then $\lambda^{Q^+}\asymp \lambda^{Q^-} \asymp \lambda^{\tilde{Q}}.$
\end{lemma}
\begin{proof}
Let $U$ open with $0<|U|<\frac{1}{2}.$ 
Then for $x,y\in U,$ $|Q(x)-Q(y)|\leq \frac{-
C}{\log(|U|)}.$ Hence $0\leq 
\log(|U|)(Q^-(U)-Q^+(U))\leq C.$ Then $|U|^{Q^+(U)}\leq |U|^{Q^-(U)} \leq e^C|U|^{Q^+(U)}.$ The result follows.
\end{proof}
The following lemma may be found in \cite{Sob} (Lemma 2.1).
\begin{lemma} If $\nu$ is Ahlfors $Q\mhyphen$regular then $Q$ is log-H\"older continuous. \label{log} \end{lemma}
\begin{proof} By Ahlfors regularity, there exists a constant $D>1$ such that $\nu(B_r(x))\leq Dr^{Q(x)}$ and $r^{Q(x)}\leq D\nu(B_r(x))$ for all $0<r\leq\diam(X),x\in X.$ Suppose $x,y\in X$ with $0<r:=d(x,y)<\frac{1}{2}.$ Say $Q(x)\geq Q(y).$ Since $Q$ is bounded, let $R<\infty$ be an upper bound for $Q,$ and let $C$ be defined by $e^C:=2^R D^2.$ Then since $B_r(y)\subset B_{2r}(x),$ $r^{Q(y)}\leq D\nu(B_r(y))\leq D\nu(B_{2r}(x))\leq D^2 2^R r^{Q(x)} = e^Cr^{Q(x)}.$ Hence $d(x,y)^{|Q(y)-Q(x)|}\geq e^{-C}.$ So $|Q(x)-Q(y)|\leq \frac{-C}{\log(d(x,y)}.$
\end{proof}

Hence, in particular, if $\nu$ is Ahlfors $Q\mhyphen$regular then $Q$ is continuous.
A Borel measure $\nu$ is outer regular if for every Borel set $A$, 
\[\nu(A)=\inf\{\nu(U)\;|\;A\subset U,\; U \mbox{ open}\}.\] It is called inner regular if for any Borel set $A$, 
\[\nu(A)=\sup\{\nu(F)\;|\;A\supset F,\;F \mbox{ compact}\}.\] The Borel measure $\nu$ is called regular if it is both outer and inner regular.
We state the following classical result. We recall the proof here for completeness, following \cite{van}.
\begin{lemma} If $\nu$ is a finite Borel measure on a compact metric space $X$ then $\nu$ is regular. If $X$ is $\sigma$-compact then $\nu$ is inner regular.\label{reg}
\end{lemma}
\begin{proof} Let $\mathscr{M}:=\{A\;| \sup\{\nu(F)\;|\;A \supset F,\; F \mbox{ closed}\}=\inf\{\nu(U)\;|\;A\subset U,\; U \mbox{ open}\}\}.$ Since $X$ is both open and closed, $X\in \mathscr{M}.$ Suppose $A\in \mathscr{M}.$ Then for $\epsilon>0$ if $F\subset A\subset U$ with $F$ closed, $U$ open, and $\nu(U\setminus F)<\epsilon,$ then $U^c\subset A^c \subset F^c,$ $U^c$ closed, $F^c$ open, and $\nu(F^c\setminus U^c)=\nu(U\setminus F)<\epsilon.$ It follows that $A^c \in \mathscr{M}.$ Let $(A_n)\subset \mathscr{M}.$ Then for $\epsilon>0$, for each $n$ let $F_n\subset A_n \subset U_n$ with $F_n$ closed, $U_n$ open, and $\nu(U_n\setminus F_n)<\frac{\epsilon}{2^{n+1}}.$ Then let $A=\cup A_n.$ Let $N$ be sufficiently large so that $\nu(\cup_{n=1}^N F_n)>\nu(\cup_n F_n)-\frac{\epsilon}{2}.$ Then let $F=\cup_{n=1}^N F_n, U=\cup U_n.$ Then $F\subset A \subset U,$ $F$ is closed, $U$ is open, and $\nu(U\setminus \cup_n F_n) = \nu(U)-\nu(\cup_n F_n)<\frac{\epsilon}{2}.$ So $\nu(U\setminus F)=\nu(U)-\nu(F)<\epsilon.$ So $A\in \mathscr{M}$. Hence $\mathscr{M}$ is a $\sigma-$algebra. Let $A\subset X$ closed. Then $\nu(A)=\sup\{\nu(F)\;|\;X \supset F \mbox{ closed }\}.$ Let $U_n=B_{\frac{1}{n}}(A).$ Then $A\subset U_n,$ each $U_n$ is open, and $\cap_n U_n = A.$ So $\inf_n \nu(U_n) = \nu(A)$ by continuity of measure, since $\nu(X)<\infty.$ Hence $\nu(A)\leq \inf\{\nu(U)\;|\;A\subset U \mbox{ open }\}\}\leq \inf_n \nu(U_n)=\nu(A).$ Hence $\mathscr{M}$ contains all Borel sets. Since $X$ is compact, all closed subsets of $X$ are compact. The first result then follows.  

Now, if $X$ is $\sigma-$compact, $X=\cup X_n$ where $X_i\subset X_j$ for $j\geq i$ and each $X_i$ is compact. Let $A\subset X$ measurable. Let $\epsilon>0.$ Let $K_1\subset X_1\cap A$ compact with $\nu(A\cap X_1)<\nu(K_1)+\epsilon.$ Then we may choose $K_2\subset X_2\cap A$ with $K_1\subset K_2$ and $\nu(A\cap X_1)<\nu(K_1)+\frac{\epsilon}{2}.$ Continuing by induction we may choose a sequence $K_n$ with $K_n\subset X_n \cap A,$ $K_n$ compact, $K_n\subset K_{n+1}$  and $\nu(A\cap X_n)<\nu(K_n)+\frac{\epsilon}{n}$ for each $n$. Then by continuity of measure, $\nu(A)=\sup_n \nu(A\cap X_n) \leq \sup _n \nu(K_n).$ The result follows.
\end{proof}

\begin{proposition} \label{equiv} If $\nu$ is a finite Ahlfors $Q\mhyphen$regular Borel measure on a separable metric space then $\nu\asymp \lambda^{Q_c}.$
\end{proposition}

\begin{proof} Since $Q$ is bounded, let $R>0$ be an upper bound for $Q$. Let $C>0$ be a constant such that $\nu(B_r(x))\leq Cr^{Q(x)}$ and $r^{Q(x)}\leq C\nu(B_r(x))$ for all $x\in X$ and $0<r\leq \diam(X).$ 
Then let $A\subset X$ be Borel measurable. For $\delta>0$ let $(B_i)_{i\in I} \in 
\mathscr{B}_\delta(A).$ Say $B_i=B_{r_i}(x_i).$ Let 
\[r_i^\prime=\sup\{d(x_i,y)\;|\;y\in B_i\}.\] Then $B_i\subset \{y\;|\;
d(x_i,y)\leq r_i^\prime\}=B_{r_i'}[x].$ Note, by Ahlfors regularity, $\nu(B_{r_i'+1}(x_i))<\infty$. So by continuity of measure \[\nu(B_{r_i'}[x_i])=\inf_{n\geq 1} \nu(B_{r_i'+\frac{1}{n}}(x_i))\leq C\inf_{n\geq 1} (r_i'+\frac{1}{n})^{Q(x_i)}=Cr_i'^{Q(x_i)}.\] It follows that \[\nu(A)\leq \sum_{i\in I}\nu(B_i)\leq 
C\sum_i r_i'^{Q(x_i)} \leq C\sum_i |B_i|^{Q(x_i)}.\] Hence \[\nu(A)\leq C\lambda_\delta^{Q_c}
(A)\leq C \lambda^{Q_c}(A).\]

Let $A$ be open. Let $\delta>0.$  Since $X$ is separable, let $(B_i)_{i\in I}\in \mathscr{B}_{\frac{\delta}
{10}}(A)$ such that $\cup B_i = A$ and $I$ is at most countable. Then by the 
Vitali Covering Lemma, there exists a disjoint sub-collection $(B_j)_{j\in J}, J\subset I,$ of the 
$(B_i)_{i\in I}$ such that $A\subset \cup_{j\in J} 5B_j.$ Let $Q_j=Q(x_j),$ where $x_j$ is 
the center of $B_j.$ Let $r_j$ be the radius of $B_j.$ Then \[\lambda_\delta^{Q_c}(A)\leq 
\sum_j |5B_j|^{Q_j}\leq 10^R\sum_jr_j^{Q_j} \leq 10^R C\sum_j 
\nu(B_j)\leq 10^R C\nu(A).\] Since $\delta>0$ was arbitrary, \[\lambda^{Q_c}(A)\leq C 10^R \nu(A).\]
Let $B\subset X$ Borel measurable. Since $\nu$ is regular, for $\epsilon>0,$ let $A\subset X$ 
open with $B\subset A$ such that $\nu(B)\geq \nu(A)-\epsilon.$ Then \[\lambda^{Q_c}(B)\leq \lambda^{Q_c}(A)\leq 10^RC(\nu(B)+\epsilon).\] Since $\epsilon>0$ is arbitrary, \[\lambda^{Q_c}(B)\leq 10^RC\nu(B).\]
\end{proof}
\begin{theorem} \label{main} Let $X$ be a compact metric space. If $\nu$ is an Ahlfors $Q\mhyphen$regular Borel measure on $X$ then $Q=\alpha$ and $\nu \asymp H_{\loc}.$ 
\end{theorem}

\begin{proof} Since $X$ is compact, $\nu$ is finite. Hence by the previous proposition \ref{equiv}, $\nu\asymp \lambda^{Q_c}.$ Hence $X$ is $Q$ amenable, and $Q$ is continuous, as it is log-H\"older continuous. By \ref{amen}, $Q=\alpha.$ 

Let $B$ be a non-empty open ball. Since $B$ is open and the definition of $\alpha$ is local, we may apply \ref{dimmax} to $B$. Hence \[\dim(B)=\sup_{x\in B} \alpha(x) = Q^+(B).\] Therefore $\lambda^{Q^+}= \lambda_{\loc}.$ Finally, since $\lambda_{\loc}\asymp H_{\loc}$ by \ref{asym} and since $\asymp$ is transitive, we have  $\nu\asymp H_{\loc}.$
\end{proof}

Hence if a compact space admits an Ahlfors $Q\mhyphen$regular Borel measure then $\alpha=Q$ and that measure is strongly equivalent to the local measure.

\section{Constructions}

We now apply the preceding  mathematical developments to several examples. The first, a Koch curve with variable local dimension, may be found in \cite{Sob}. The remaining examples, a Sierpinski gasket of variable local dimension, a Sierpinski carpet of variable local dimension, and a Vicsek tree of variable local dimension, have not been considered, at least to our knowledge, in the literature before. 

\begin{subsection}{A variable dimensional Koch curve}
 Note that this first example is not new. It is a particular case of Koch curve constructed and analyzed in \cite{Sob}. Moreover, the idea for a Koch curve of variable dimension may be found in \cite{nottale}, although the notion of dimension used there is not precise. 

 Given $0<\theta_1<\theta_2<\frac{\pi}{2},$ we construct a compact metric space $K$ and a continuous bijective map $\phi:[0,1]\rightarrow K$ such that if $x\in K$ and $\phi(t)=x$ then \[\alpha(x)=\frac{2\log(2)}{\log(2+2\cos(\theta_1+t(\theta_2-\theta_1)))}.\] Then, for example (see figure 2), with $\theta_1=\frac{\pi}{36}=5^\circ,\theta_2=\frac{4\pi}{9}=80^\circ,$ the local dimension $\alpha$ variously continuously from approximately $1.001$ to approximately $1.625$. 
 
It also holds that the local Hausdorff measure $H_{\loc}$ is Ahlfors $\alpha\mhyphen$regular on $K$ and
$0<H_{\loc}(K)<\infty.$ Moreover, if $c\geq 0$ is any constant then $H^c(K)$ is either $0$ or $\infty.$ Hence the local Hausdorff measure is in a sense the ``correct" measure for $K$ \cite{Sob}.

The construction is a particular case of the construction given in \cite{Sob}. A reader interested in further details and proofs should consult \cite{Sob}. 

By a generator of parameters $L$ and $\theta$ we mean the following. Label the segments $1$ through $4.$ Put segment $1$ along the positive $x\mhyphen$axis with one end at the origin. Then connect segment $2$ at an angle of $\theta$ with the positive $x\mhyphen$ axis to the endpoint of segment $1$ at $(L,0)$. Then the other endpoint of segment $2$ will lie at $(L+L\cos(\theta),L\sin(\theta)).$ Then connect one endpoint of segment $3$ to the point $(L+L\cos(\theta),L\sin(\theta))$ in such a way that its second endpoint lies at $(L+2L\cos(\theta),0)$ on the positive $x\mhyphen$axis.  Then place segment $4$ on the positive $x\mhyphen$axis with its endpoints at $(L+2L\cos(\theta),0)$ and $(2L+2L\cos(\theta),0)$. We use coordinates only to specify the construction. The element may be translated and rotated freely. Note that if the overall length $2L+2L\cos(\theta)$ is set to $1, $ then $L$ and $\theta$ are related by 
$L=\frac{1}{2+2\cos(\theta)}.$ 
\begin{figure}[H]
\centering
\includegraphics[scale=.25]{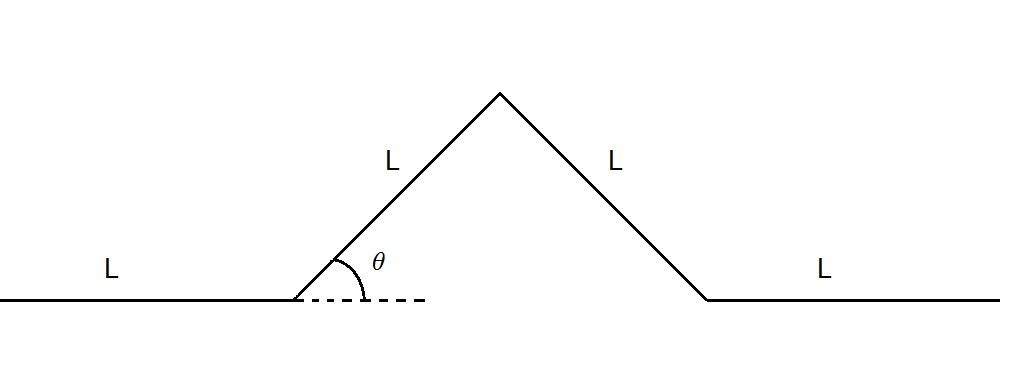}
\caption{A generator of parameters $L$ and $\theta.$}
\end{figure}

We then construct the curve recursively as follows. Let $M>0$ and $\theta_1,\theta_2$ lower and upper angles with $0<\theta_1< \theta_2<\frac{\pi}{2}.$ For stage $0$ let $K_0$ be the segment connecting $(0,0)$ and $(M,0).$ At stage $1$, form $K_1$ by replacing $K_0$ with the generator of parameters $L$ and $\frac{\theta_1+\theta_2}{2}$ where $L=\frac{M}{2+2\cos((\theta_1+\theta_2)/{2})}.$ Then suppose we have constructed $K_0,K_1,...,K_n$ for any $M>0$ and $0<\theta_1<\theta_2<\frac{\pi}{2}$ where $K_n$ is made of 4 versions of $K_{n-1}$ each of overall length $L$, where the $i\mhyphen$th, for $i=1,2,3,4$ labeled from left to right, has lower and upper angles $\theta_1+(i-1)\frac{(\theta_2-\theta_1)}{4}$ and $\theta_1+i\frac{(\theta_2-\theta_1)}{4}$, respectively. Then, for $M>0,0<\theta_1<\theta_2<\frac{\pi}{2},$ form $K_{n+1}$ by replacing the $i\mhyphen$th version of $K_{n-1}$ in $K_n$ with a $K_n$ of overall length $L$ with lower and upper angles $\theta_1+(i-1)\frac{(\theta_2-\theta_1)}{4}$ and $\theta_1+i\frac{(\theta_2-\theta_1)}{4}$, respectively. Hence by induction we may construct $K_n$ for any $n$ and any overall length $M>0$ and angles of interpolation $0<\theta_1<\theta_2<\frac{\pi}{2}.$ 

We will now fix $M=1$ and the endpoints of $K_n$ at $(0,0)$ and $(1,0).$ 
Then note $K_n$ is made up of $4^n$ segments and $\phi_n$ is the map that is the bijective, piecewise continuous map that is constant speed from $\frac{i-1}{4^k}$ to $\frac{i+1}{4^k}$ connecting the endpoints of the $i\mhyphen$th segment making up stage $n,$ with $\phi_n(0)=(0,0)$ and $\phi_n(1)=(1,0).$ Then, as the sequence of $\phi_n$ is Cauchy in $C([0,1],\mathbb{R}^2),$ it converges to a continuous $\phi.$ We let $K=\phi[0,1].$ Moreover, the sequence of $K_n$ is Cauchy in the Hausdorff metric on compact subsets of $[0,1]\times [0,1],$ which is known to be complete. Hence we may also define $K$ as the limit of the $K_n$ in the sense of Hausdorff convergence.

\begin{figure} \label{fig2}
\centering
\includegraphics[scale=.25]{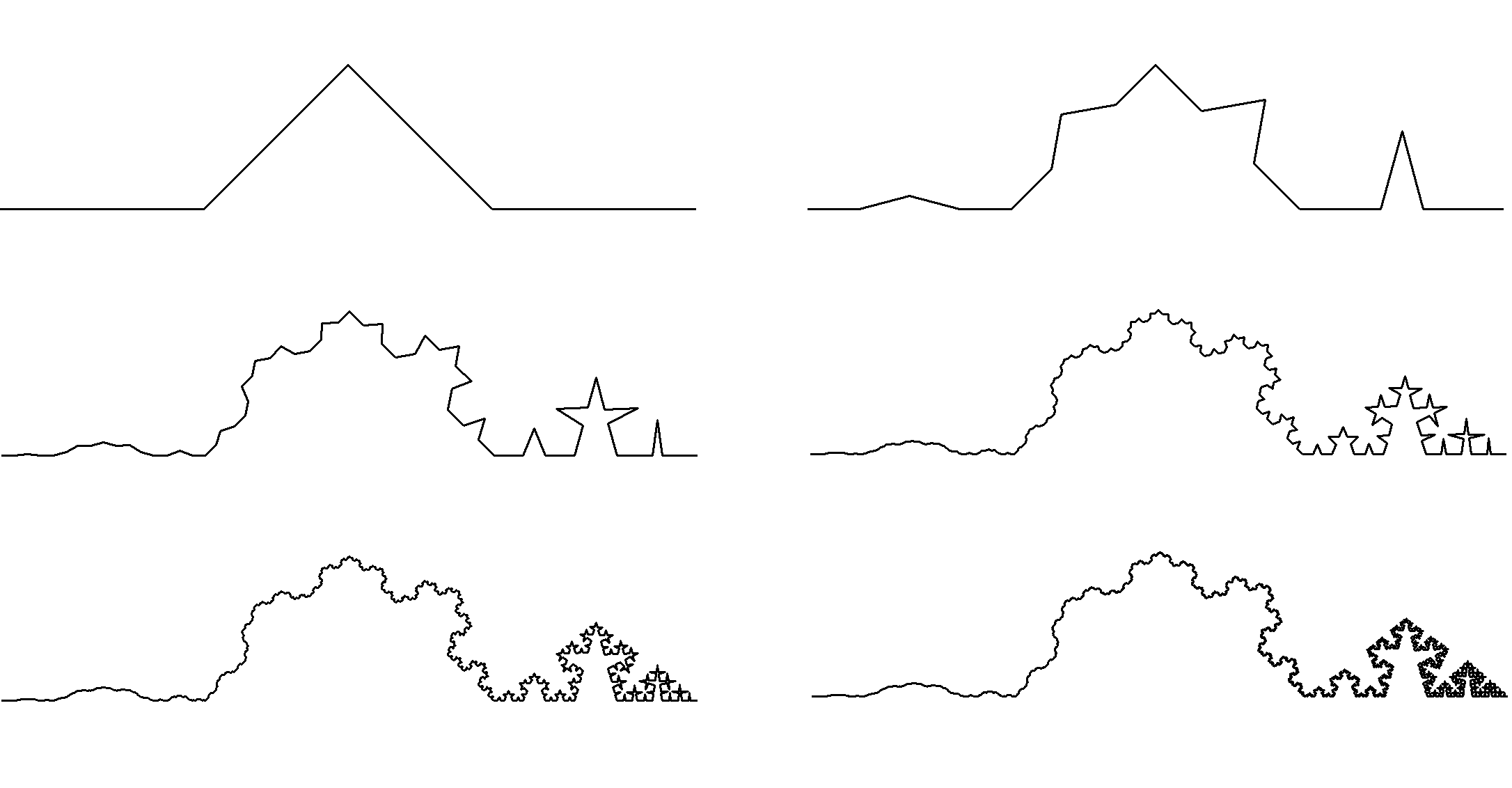}
\caption{Stages 1 through 6 of a variable dimensional Koch curve where $\theta_1=5^\circ$ and $\theta_2=80^\circ.$}
\end{figure}

We state the following proposition without proof since the it follows from \cite{Sob} Theorem 3.4, as the function $s:[0,1]\rightarrow (1/4,1/2)$ defined by $s(t)=\frac{1}{2+2\cos(\theta_1+t(\theta_2-\theta_1))}$ is Lipschitz.

\begin{proposition} Let $Q(x)=\frac{2\log(2)}{\log(2+2\cos(\theta_1+\phi^{-1}(x)(\theta_2-\theta_1)))}$. Then $\lambda^{Q_c}$ is an Ahlfors $Q\mhyphen$regular measure for $K.$ 
\end{proposition}

\begin{corollary} $Q$ is the local dimension and $H_{\loc}$ is Ahlfors $\alpha\mhyphen$regular. In particular, $0<H_{\loc}(K)<\infty.$ Moreover, if $c\geq 0$ then $H^c(K)$ is $0$ or $\infty.$

\end{corollary}
\begin{proof}
The first claim follows immediately from Theorem \ref{main}. The last follows from Proposition \ref{3.6}.
\end{proof}
\end{subsection}
\subsection{A variable dimensional gasket}

Given a length $L$ and a scaling parameter $r\in [0,1/2],$ one may construct a generator of parameters $L$ and $r$ by exercising from a filled equilateral triangle of side length $L$ everything but the the three corner equilateral triangles with side length $rL$,  as may be seen in Figure 3 below.

\begin{figure}[H]
\centering
\includegraphics[scale=.25]{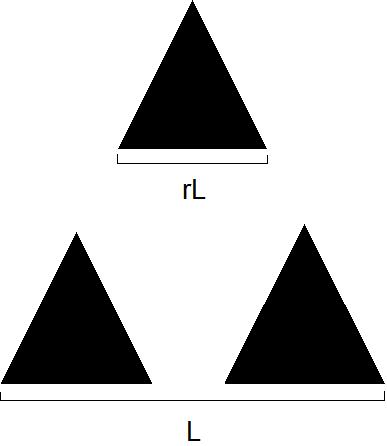}
\caption{A generator of parameters $L$ and $r.$}
\end{figure}

We then construct the gasket recursively as follows. Let $L>0$ be the length parameter and $r_1,r_2$ lower and upper ratios with $0\leq r_1< r_2\leq 1/2.$ For stage $0$ let $K_0$ be the filled in equilateral triangle in $\mathbb{R}^2$ of side length $L$ with vertices at $(0,0), (L,0),$ and $(L/2,\sqrt{3}L/2)$. At stage $1$, form $K_1$ by replacing $K_0$ with the generator of parameters $L$ and $\frac{r_1+r_2}{2}$. Then suppose we have constructed $K_0,K_1,...,K_n$ for any $L>0$ and $0\leq r_1< r_2\leq\frac{1}{2}$ where $K_n$ is made of 3 versions of $K_{n-1}$ each of overall length $L(\frac{r_1+r_2}{2})$. Label these versions by a parameter $i$ for $i=1,2,3$ where version $1$ occupies the bottom left triangle and version $2$ is on the bottom to the right and version $3$ is on the top. Then, for $L>0,0\leq r_1< r_2\leq \frac{1}{2},$ form $K_{n+1}$ by replacing the $i\mhyphen$th version of $K_{n-1}$ in $K_n$ with a $K_n$ of length $\frac{r_1+r_2}{2}L$ and lower and upper ratios $r_1+(i-1)\frac{(r_2-r_1)}{3}$ and $r_1+i\frac{(r_2-r_1)}{3}$. Hence by induction we may construct $K_n$ for any $n$ and any overall length $L>0$ and ratios $0\leq r_1< r_2\leq \frac{1}{2}.$ 
Note that since $K_{n+1}\subset K_n$ for each $n$ and each $K_n$ is a closed subset of a compact set in $\mathbb{R}^2$, we may define the compact gasket $K$ by $K=\cap_{n=0}^\infty K_n.$

The following figure (Figure 4) was constructed with $r_1=.4$ and $r_2=.5.$ 
\begin{figure}[H]
\centering
\includegraphics[scale=.55]{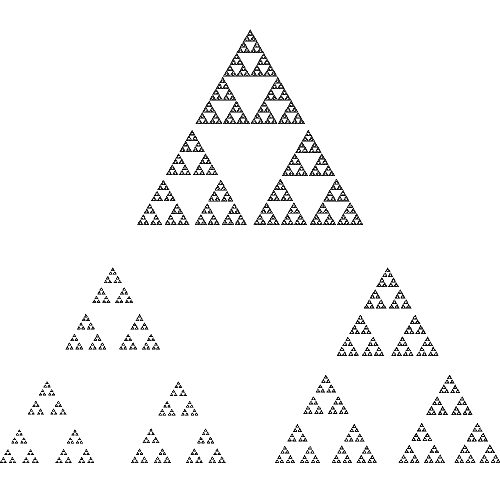}
\caption{A variable dimensional gasket with $r_1=.4, r_2=.5.$}
\end{figure}

\subsection{A variable dimensional carpet}
Given a base length $b$, a height $h$, and a scaling parameter $r\in [0,1],$ one may construct a generator of parameters $b,h,$ and $r$ by exercising from a filled rectangle of base $b$ and height $h$ the center open rectangle of base $br$ and height $hr.$ See the figure below.
\begin{figure}[H]
\centering
\includegraphics[scale=.3]{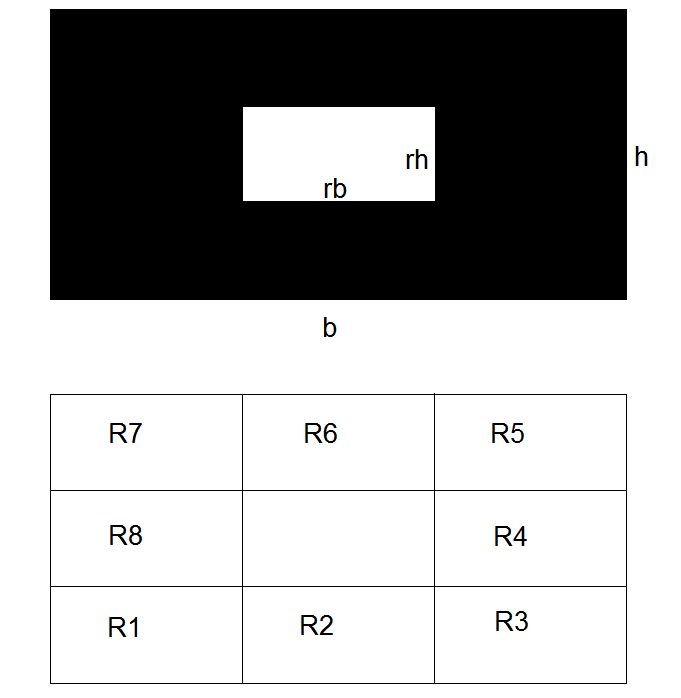}
\caption{(top) A generator of parameters $b,h,$ and $r$ and (bottom) its decomposition into sub-rectangles.}
\end{figure}
As illustrated in the figure, we may decompose the generator into $8$ sub-rectangles $R1,R2,...,R_8.$ Rectangles $R1,R3,R5,$ and $R7$ have base $\frac{b-rb}{2}$ and height $\frac{h-rh}{2}.$ Rectangles $R2$ and $R6$ have base $rb$ and height $\frac{h-rh}{2}.$ Rectangle $R4$ and $R8$ have base $\frac{b-rb}{2}$ and height $rh.$

We then construct the carpet recursively as follows. Let $b, h>0$ be base and height parameters and $r_1,r_2$ lower and upper ratios with $0\leq r_1< r_2\leq 1.$ For stage $0$ let $K_0$ be the filled in rectangle in $\mathbb{R}^2$ with vertices at $(0,0), (b,0), (b,h),$ and $(0,h)$. At stage $1$, form $K_1$ by replacing $K_0$ with the generator of parameters $b, h,$ and $\frac{r_1+r_2}{2}$. Then suppose we have constructed $K_0,K_1,...,K_n$ for any $L>0$ and $0\leq r_1< r_2\leq 1$ where $K_n$ is made of 8 versions of $K_{n-1}$ where version $i$ sits in the spot for sub-rectangle $Ri$ of the generator for $i=1,2...,8.$ Then, for $L>0,0\leq r_1< r_2\leq 1,$ form $K_{n+1}$ by replacing the $i\mhyphen$th version of $K_{n-1}$ in $K_n$ with a $K_n$ of base and height equal to the base and height of sub-rectangle $Ri$ of $K_1$ and lower and upper ratios $r_1+(i-1)\frac{(r_2-r_1)}{8}$ and $r_1+i\frac{(r_2-r_1)}{8}$, respectively, for $i=1,...,8$. Hence by induction we may construct $K_n$ for any $n$ and any base $b$, height $h,$ and ratios $0\leq r_1< r_2\leq 1.$ 
Note that since $K_{n+1}\subset K_n$ for each $n$ and each $K_n$ is a closed subset of a compact set in $\mathbb{R}^2$, we may define the compact carpet $K$ by $K=\cap_{n=0}^\infty K_n.$

The following figure (Figure 6) was constructed with $b=h$ and $r_1=1/6,r_2=1/2.$
\begin{figure}[H]
\centering
\includegraphics[scale=.75]{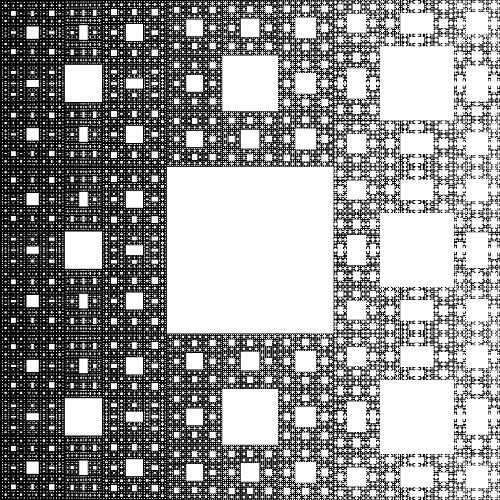}
\caption{A variable dimensional carpet constructed from $b=h$ and $r_1=1/6, r_2=1/2$.}
\end{figure}

\subsection{A variable dimensional Vicsek tree}
Given a length $L$ and a scaling ratio $r\in [0, 1],$ one may construct a generator with parameters $L$ and $r$ as follows. Given a filled in square of side length $L$, drawing a square in the middle of side length $Lr$ induces a decomposition of the initial square into $9$ sub-rectangles. Label the rectangles from bottom left $R1,R2,R3$ on the bottom row, $R4,R5,R6$ on the middle row, and $R7,R8,R9$ on the top row. Then remove the interiors of $R2,R4, R6,$ and $R8.$ We emphasize that the ambient space is homeomorphic to $[0,L]\times [0, L]$, and the interior operation is taken with respect to the releative topology of this space. Hence the interior of the outer rectangles that are removed includes part of the outer ``boundary" of the original square; so the result is a union of 5 disjoint closed squares. See the figure below. 

\begin{figure}[H]
\centering
\includegraphics[scale=.22]{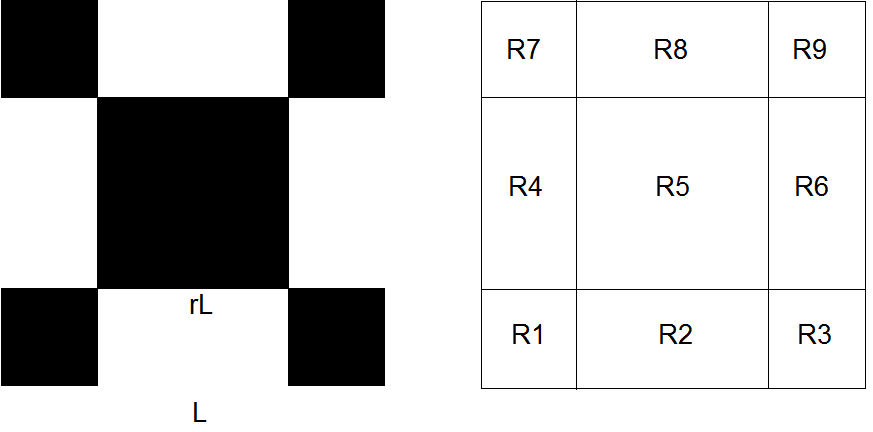}
\caption{(left) A generator of parameters $L$ and $r$ and (right) labeling.}
\end{figure}

We then may perform the construction of the tree as follows. Given a length $L$ and ratios $r_1, r_2$ with $0\leq r_1< r_2\leq 1,$ let $K_0$ be the square $[0,L]\times [0,L]$. Then let $K_1$ be the result of replacing $K_0$ with the generator with parameters $L$ and $\frac{1}{2}(r_1+r_2).$  Having constructed $K_0, K_1,,,.K_n$ for some $n\geq 1,$ where $K_n$ is made up of $5$ copies of $K_{n-1}$ in rectangles $R1, R3, R5, R7,$ and $R9,$ 
we construct $K_{n+1}$ as follows. Replace $R1$ and $R7$ with copies of $K_n$ of length $\frac{1}{2}(L-rL)$ and ratios $r_1$ and $r_1+\frac{1}{3}(r_2-r_1).$ Replace $R5$ with a $K_n$ of length $rL$ and ratios $r_1+\frac{1}{3}(r_2-r_1)$ and $r_1+\frac{2}{3}(r_2-r_1).$ Replace $R3$ and $R9$ with copies of $K_n$ of length $\frac{1}{2}(L-rL)$ and ratios $r_1+\frac{2}{3}(r_2-r_1)$ and $r_2.$
Then let $K=\cap_{n=0}^\infty K_n$. 

\begin{figure}[H]
\centering
\includegraphics[scale=.5]{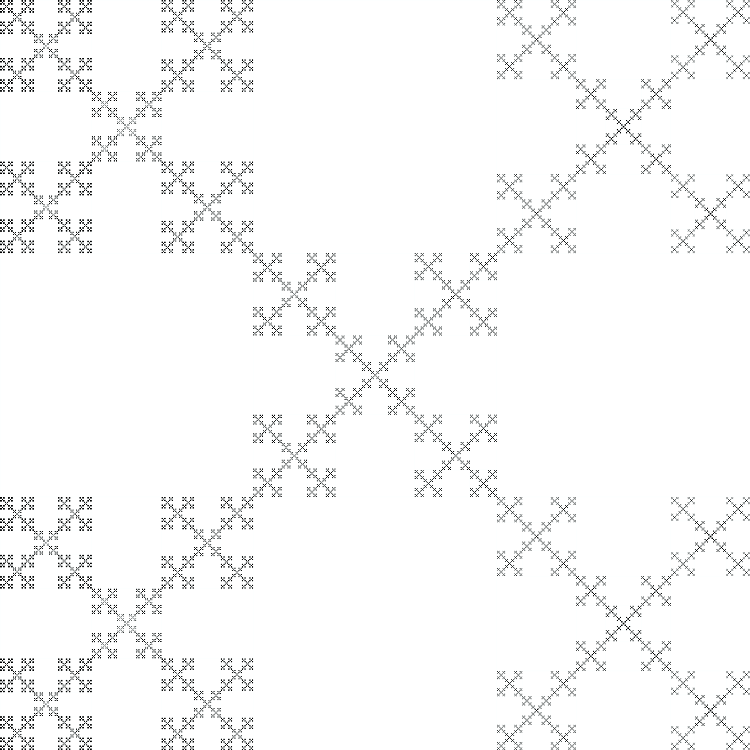}
\caption{Stage $6$ of a variable dimensional Vicsek tree constructed with $r_1=.25, r_2=.4$.}
\end{figure}
See also \cite{telcs2006} for the example of a weighted infinite Vicksek tree graph. 
\section{The definition of $\beta$}
In this section we give a particular definition of a scaling exponent $\beta$ that works for any compact metric space. However, we strongly emphasize that the essential idea behind the construction may be applied in much greater generality. We will see this in the later sections where we, as a preliminary route of investigation, instead of approximating by discrete sets, adopt a continuous space approach to approximating random walks using a given measure. There we will adapt the definition of $\beta$ in a natural way to that setting. 

Recall an $\epsilon$-net on a metric space $(X,d)$ is a subset $N\subset X$ such that $\cup_{x\in N} B_\epsilon(x)=X$ and if $x,y\in N$ with $x\neq y$ then $d(x,y)\geq \epsilon.$
The following proposition may be proven by a straightforward application of Zorn's Lemma and the definition of compactness. As such, we omit the proof.

\begin{proposition} 
Let $X$ be a metric space. Then for any $\epsilon>0$ there exists an $\epsilon$-net in $X.$ Moreover, if $\epsilon>0$ and $A\subset X$ such that $d(x,y)\geq \epsilon$ for $x,y\in A$ with $x\neq y,$ then there exists an $\epsilon$-net $N$ in $X$ with $A\subset N.$ In particular if $0<\epsilon'<\epsilon$ and $N$ is an $\epsilon$-net in $X$ then $N$ may be extended to an $\epsilon'$-net $N'$ in $X.$ If $X$ is compact then any $\epsilon$-net in $X$ is finite.
\end{proposition}

By a (simple) graph, we mean a pair $G=(V,E)$ of sets with $E\subset V\times V$ such that if $(x,y)\in E$ then $(y,x)\in E$. If $(x,y)\in E$ we write $x\sim y$ or $x\sim_G y$ if we wish to emphasize the dependence on $G.$ The set $V$ is called the set of vertices and the set $E$ is called the set of (undirected) edges. 

If $G=(V,E)$ is a graph and $x\in V$ then the degree of $x,$ $\deg_G(x)$, is defined by $\deg_G(x)=\#\{y\in V\;|\;y\sim_G x\}.$ Note that $y\in \mathbb{Z}_+\cup\{\infty\}.$ 

Given an $\epsilon$-net $N$ in a metric space $X$ we may define an approximating graph with vertex set $N$ in a number of ways. We outline two of these methods below.

First, given a parameter $\eta\geq 1$ and an $\epsilon$-net $N$ we define a ``covering graph" as follows. Given an $\epsilon$-net $N$ in $X$, we define a graph $G(N)$ to have vertex set $N$ and edge set $E(N)$ defined by $(x,y)\in E(N)$ if and only if $x,y\in V$ and $B_{\eta\epsilon}(x)\cap B_{\eta\epsilon}(y)\neq \emptyset.$  This is the approach taken in, for example, \cite{sturm}.

Alternatively, given a parameter $\rho\geq 2$ and an $\epsilon$-net $N$ one may define a ``proximity graph" as follows. Given an $\epsilon$-net $N$ in $X$, we define a graph $G(N)$ to have vertex set $N$ and edge set $E(N)$ defined by $(x,y)\in E(N)$ if and only if $x,y\in V$ and $d(x,y)<\rho\epsilon.$  This is the approach taken in, for example, \cite{burago2013graph}.

Now suppose $X$ is compact. Note that if $N$ is an $\epsilon$-net, for either of the above approaches for constructing $G(N)=(N,E(N))$ we have $0<\deg_{G(N)}(x)<\infty$ for all $x\in N$ since $(x,x)\in E(N)$ for all $x \in N.$ Given an $\epsilon$-net $N$ in $X$ and a graph $G(N)$ such that $0<\deg(x)<\infty$ for all $x\in N$, we construct a random walk on $N$ and expected exit times from subsets in a standard way (see, for instance, \cite{art}). For $x,y\in N$, we define a transition probability $p(G(N))_{x,y}$ to jump from $x$ to $y$ in one time step by 
\[p(G(N))_{x,y}=\frac{1}{\deg_{G(N)}(x)}\chi_{\{z\in V\;|\;z\sim_{G(N)} x\}}(y).\] This defines a discrete time Markov process $(Y(N)_{k})_{k\in \mathbb{Z}_+}$ with finite state space $N$. Given $x\in N,$ let $\mathbb{P}^x$ be the probability defined on all paths starting at $x$, and let $\mathbb{E}^x$ be the expectation with respect to $\mathbb{P}^x.$ 

Now, given a set $A\subset X,$ we let \[\tau(N)_A:=\inf\{k\in\mathbb{Z}_+\;|\;Y(N)_k\notin N\cap A\},\] where we adopt the convention that $\inf \emptyset = \infty.$

Then we define the \textit{expected exit time} (see also \cite{art}) from $A$ starting at $x\in N$, denoted $E_{N,A}(x),$ as follows\[E_{N,A}(x)=\mathbb{E}^x\tau(N)_A.\] We let \[E^+_{N,A}:=\max_{x\in N}E_{N,A}(x)=\max_{x\in N\cap A}E_{N,A}(x).\]

We now are in a position to define the central notion of $\beta.$ 

Given $A\subset X$ and $\gamma\geq 0,$ let 
\[\omega_{\gamma}(A):=\inf_{\delta>0}\sup\{E^+_{N,A}\epsilon^\gamma\;|\;N\;\mbox{is an\;}\epsilon\mhyphen\mbox{net with\;}0<\epsilon<\delta\}.\]
\begin{proposition} For any $A\subset X,$ $\sup\{\gamma\geq 0\;|\;\omega_{\gamma}(A)=\infty\}=\inf\{\gamma\geq 0\;|\;\omega_{\gamma}(A)=0\}.$ We call the common value $\beta(A).$ Then if $A'\subset A$ we have $\beta(A')\leq \beta(A).$\end{proposition}

\begin{proof}
Suppose $0\leq \gamma'<\gamma.$ Then \[\sup\{E^+_{N,A}\epsilon^\gamma\;|\;N\;\mbox{an\;}\epsilon\mhyphen\mbox{net with\;}0<\epsilon<\delta\}\leq \delta^{\gamma-\gamma'}\sup\{E^+_{N,A}\epsilon^{\gamma'}\;|\;N\;\mbox{an\;}\epsilon\mhyphen\mbox{net with\;}0<\epsilon<\delta\},\] for any $\delta>0.$ Hence if $\omega_{\gamma'}(A)<\infty,$ then $\omega_{\gamma}(A)=0;$ and if $\omega_{\gamma}(A)>0,$ then $\omega_{\gamma'}(A)=\infty.$ Let $\beta(A):=\sup\{\gamma\geq 0\;|\;\omega_{\gamma}(A)=\infty\}.$ If $0\leq\gamma\leq \beta(A)\leq \infty,$ then $\omega_{\gamma}(A)=\infty$. Hence $\inf\{\gamma\geq 0\;|\;\omega_{\gamma}(A)=0\}\geq \beta(A).$ If $\inf\{\gamma\geq 0\;|\;\omega_{\gamma}(A)=0\}> \beta(A),$ then there exists a $\gamma>\beta(A)$ with $\omega_{\gamma}(A)>0.$ But then if $\beta(A)<\gamma'<\gamma,$ then $\omega_{\gamma'}(A)=\infty$, a contradiction. It follows that $\beta(A)=\inf\{\gamma\geq 0\;|\;\omega_{\gamma}(A)=0\}$.

Now suppose $A'\subset A.$ Let $N$ be an $\epsilon$-net. Clearly for any infinite $G(N)$ path $\omega$, $\tau(N)_{A'}(\omega)\leq \tau(N)_{A}(\omega).$ Hence $E_{N,A'}(x)\leq E_{N,A}(x)$ for all $x\in N.$ It follows, for any $\gamma\geq 0,$ that $\omega_{\gamma}(A')\leq \omega_{\gamma}(A).$ Hence \[\{\gamma\geq 0\;|\;\omega_{\gamma}(A')= \infty\}\subset\{\gamma\geq 0\;|\;\omega_{\gamma}(A)=\infty\}.\] Therefore $\beta(A')\leq \beta(A),$ as desired.
\end{proof}

We now define the local exponent $\beta:X\rightarrow [0,\infty]$ as follows.
\begin{definition} For $x\in X$ let \[\beta(x):=\lim_{r\rightarrow 0^+}\beta(B_r(x))=\inf_{r>0}\beta(B_r(x)).\]
\end{definition}

We now give a number of remarks. First, note that for each $A\subset X$ we have defined $\beta(A)$ by means of a ``$\limsup$" along $\epsilon$-nets. It may not be surprising that we also could have chosen to use a ``$\liminf$". That is, for $\gamma\geq 0$ we may define \[\omega^{-}_{\gamma}(A):=\sup_{\delta>0}\inf\{E^+_{N,A}\epsilon^\gamma\;|\;N\;\mbox{is an\;}\epsilon\mhyphen\mbox{net with\;}0<\epsilon<\delta\}.\] Then we may define 
\[\beta^{-}(A):=\sup\{\gamma\geq 0\;|\;\omega^{-}_{\gamma}(A)=\infty\}=\inf\{\gamma\geq 0\;|\;\omega^{-}_{\gamma}(A)=0\}.\]
Then as before one may show that this is well defined and that if $A'\subset A$ then $\beta^{-}(A')\leq \beta^{-}(A).$ Hence we have that \[\beta^{-}(x):=\lim_{r\rightarrow 0^+}\beta^{-}(B_r(x))\] exists. Then since $\omega^{-}_\gamma(A)\leq\omega_{\gamma}(A),$ we have that $\beta^{-}(A)\leq \beta(A).$ Hence \[\beta^{-}\leq \beta.\] In this paper, however, we will focus on $\beta$.

One justification for our choice of the definition of $\beta$ is as follows. Intuitively, $\beta$ may be seen as a localized ``walk packing dimension." That is, if one has a continuous curve $\gamma:[0,T]\rightarrow X$ and $\gamma^*$ is its image in $X$, the packing pre-measure of $\gamma^*$ at a given dimension $d$ is defined as a $\limsup$ as $\delta\rightarrow 0^+$ of terms of the form $\sum|B_i|^d$ where the $B_i$ are a collection of disjoint balls with centers in $\gamma^*$ and diameters less than $\delta.$ Given an $\epsilon<\delta$ and an $\epsilon$-net on $\gamma^*$ the balls of radius $\epsilon/2$ with centers elements of the net are then an allowed packing. Since the diameter of each ball of radius $\epsilon/2$ is bounded above by $\epsilon$ and below by $\epsilon/2$ one might consider, as an approximation, terms of the form $\#N\epsilon^d$ where $\#N$ is the number of elements of $N$. If $B$ is a ball with center $x_0$ and $\gamma:[0,\infty)\rightarrow X$ is a continuous ``sample path" with $\gamma(0)=x_0,$ set $T=\inf\{t\geq 0\;|\;\gamma(t)\notin B\}$. Then $\gamma|_{[0,T]}^*$ is the image of the path in $B$. If $N$ is an $\epsilon$-net on $\gamma|_{[0,T]}^*$ containing $x_0$, and if one considers $N$ as an approximation to the continuous sample 
path, then the the $(\#N-1)\mhyphen$th step along the approximation is the last step before exiting $B.$ Therefore, one may intuitively regard $\beta$ as a kind of local walk packing dimension. In that interpretation, the use of $\limsup$ is entirely natural.

In practice, such as is the case with many of the standard examples of self-similar fractals such as the Sierpinski gasket or carpet, arising from the construction one may be given a sequence $\epsilon_k$ of positive numbers decreasing to $0$ and for each $k$ an $\epsilon_k$-net $N_k$ in $X$. Then one might wish to consider a weaker definition of a $\beta$ defined only as a $\limsup$ of weighted mean exit times along the chosen sequence of nets $(N_k)_{k=1}^\infty.$

\begin{exmp} Consider a variable dimensional Koch curve $K$ considered in Section 4. Let $\alpha:X\rightarrow [0,\infty)$ be the local dimension. Let $\mu$ be the local Hausdorff measure as defined in Section 2. Then $\alpha$ is continuous and $\mu$ is Ahlfors $\alpha$-regular. Let $C\geq 1$ such that for any $x\in K$ and any $0<r<1/2$, 
$\frac{1}{C}r^{\alpha(x)}\leq \mu(B_r(x))\leq Cr^{\alpha(x)}.$ Let $B=B_r(x)$ with $0<r<1/2.$ Then let $N$ be an $\epsilon$-net in $K$ with $0<\epsilon<r.$ We use $\mu$ to estimate the number of points in $N\cap B.$ Let $\alpha^{-}(B)=\inf_{y\in B}\alpha(y)$ and $\alpha^+(B)=\sup_{y\in B}\alpha(y).$ Note that 
\[\frac{1}{C}\#(N\cap B)(\frac{\epsilon}{2})^{\alpha^+(B)}\leq \sum_{y\in N\cap B}\mu(B_{\epsilon/2}(y))\leq \mu(B)\leq \sum_{y\in N\cap B}\mu(B_{\epsilon}(y))\leq C\#(N\cap B)\epsilon^{\alpha^{-}(B)}.\] 
Hence, using the Ahlfors regularity, 
\[\frac{r^{\alpha(x)}}{C^2\epsilon^{\alpha^{-}(B)}}\leq \#(N\cap B) \leq \frac{2^{\alpha^+(B)}C^2r^{\alpha(x)}}{\epsilon^{\alpha^+(B)}}.\]

Let $n>2$ an integer and $G$ be the path graph with $n$ vertices and self-loops. That is $V=\{0,1,...n\}$ and $(j,k)\in E$ if and only if $|j-k|\leq 1.$ Then with the standard random walk on $G,$ let $E_n(x)$ be the expected number of steps needed for a walker starting at $x\in V$ to reach $0$ or $n.$ Then $E_n(k)=\frac{3}{2}n(n-k).$ This follows since for $0<k<n$ with $k\in V,$ $\frac{2}{3}E_n(k)=1+\frac{1}{3}E_n(k-1)+\frac{1}{3}E_n(k+1)$, and $E_n(0)=E_n(n)=0.$ Hence \[\frac{3n(n-1)}{4}\leq \max_{k\in V}E_n(k)\leq \frac{3n(n+1)}{4}.\]

Suppose given an $\epsilon$-net $N$ in $K$ we define the graph $G(N)$ as a covering graph defined by the edge relation $x\sim y$ if $B_{\epsilon}(x)\cap B_{\epsilon}(y)\neq \emptyset.$ Then the portion of the graph $G(N)$ in $B$ is a path graph. Hence there exists a constant $D>1$, independent of $r$, such that 
\[\frac{r^{2\alpha(x)}}{D\epsilon^{2\alpha^{-}(B)}}\leq E^+_{N, B} \leq \frac{Dr^{2\alpha(x)}}{\epsilon^{2\alpha^+(B)}}.\]
It follows that $2\alpha^{-}(B_r(x))\leq \beta(B_r(x))\leq 2\alpha^+(B_r(x)).$ Since $\alpha$ is continuous, letting $r\rightarrow 0^+$ yields \[\beta(x)=2\alpha(x).\] Hence $\beta$ is variable as well. 
This example also illustrates a substantial difference with the definition of $\beta$ often defined for infinite graphs. Barlow has shown that such an exponent is bounded below by $2$ and above by $\alpha+1$ \cite{barlowesc}. However, in the above example, $\beta=2\alpha>\alpha+1$. Intuitively, such a discrepancy comes from the fact that the graph distance does not well approximate the metric on $K.$

\end{exmp}

For the remainder of the paper, however, we shall consider continuous space approximate random walks.   
\section{Time Scale Re-Normalization of Approximating Random Walks }
By a \textit{metric measure space} we mean a triple $(X,d,\mu)$ consisting of a metric space $(X,d)$ together with a non-negative Borel measure $\mu$ on $X$ of full support. For example, if $(X,d)$ is variable Ahlfors regular, $\mu$ could be the local Hausdorff measure.
For the remainder of the paper, we assume $(X,d,\mu)$ is metric measure space where $(X,d)$ is a connected, compact metric space containing more than a single point. 

Suppose $p:X\times X\rightarrow [0,\infty)$ is a measurable function satisfying $\int_{X\times X} |p(x,y)|^2 d(\mu \otimes \mu)(x,y)<\infty$ such that for all $x\in X,$ $\int p(x,y) d\mu(y)=1.$ Then for $f\in L^2(X,\mu),$ let \[Pf(x):=\int p(x,y)f(y)d\mu(y).\] Then $P$ is a compact operator and defines a discrete time Markov process. Suppose further there exists a bounded measurable function $\phi:X\rightarrow [0,\infty)$ with $\int \phi d\mu =1$ such that \[\phi(x)p(x,y)=\phi(y)p(y,x).\] Then let a measure $\nu$ be defined by \[d\nu = \phi d\mu.\] Then $\nu$ is an equilibrium measure for the process defined by $P.$ Let us denote the discrete time Markov process induced by $P$ as $(Y_k)_{k=0}^\infty.$

For what follows let us fix a non-trivial closed ball $B:=B_R[a]:=\{x\;|d(a,x)\leq R\}$. Define $\tau_B:=\inf \{k\geq 0\;| Y_k\notin B\}.$ Note for each $x\in X$ there is a probability measure $\mathbb{P}^x$ on the space of discrete paths $\Omega_x=\{x\}\times \prod_{k=1}^\infty X.$ We denote $\mathbb{E}^x$ as the expectation with respect to this measure. 
Let $E_B(x):=\mathbb{E}^x\tau_B.$ 

\begin{proposition} \label{markov} For $x\in B$ we have \[E_B(x)=1+\int p(x,y)E_B(y)d\mu(y).\] If $x\notin B$ we have $E_B(x)=0.$
\end{proposition}

\begin{proof} The second claim is clear. If $x\notin B$ and $\omega$ is a path starting at $x,$ then $\omega(0)=x.$ Hence $\tau_B(\omega)=0.$ So $E_B(x)=\mathbb{E}^x\tau_B =0.$

First, it is not difficult to see that the map $x\mapsto E_B(x)$ is measurable. Now suppose $x\in B.$ Then $\mathbb{E}^x \tau_B = \mathbb{E}^x(1+(\tau_B-1))=1+\mathbb{E}^x(\tau_B-1).$ Hence we concentrate on the $\mathbb{E}^x(\tau_B-1)$ term. Note that if $\omega$ is a path starting at $x$ and $\omega(1)\notin B$ then $(\tau_B-1)(\omega) =0.$ So when taking the expectation we may assume the path takes its first step in $B.$ For $k\geq 1$, $y\in B,$  let \[A_k(y):=\{\omega\;|\;\omega(0)=y,\;\omega(j)\in B\;\mbox{for $1\leq j\leq k-1\;$ and \;} \omega(k)\notin B\}.\] Since we assume the first step is in $B,$ $A_1=\emptyset.$ Set $x:=x_0.$ We take the empty product to be $1.$ Then for $k\geq 2,$
\[\mathbb{P}^x(A_k(x))=\int_B p(x_0,x_1)\prod_{j=2}^{k-1} (\int_B p(x_{j-1},x_j))\int_{B^c} p(x_{k-1},x_k)d\mu(x_k)...d\mu(x_1).\] However, we recognize in the last expression, by the Markov property, for each $x_1\in B,$ \[\prod_{j=2}^{k-1} (\int_B p(x_{j-1},x_j))\int_{B^c} p(x_{k-1},x_k)d\mu(x_k)...d\mu(x_2)= \mathbb{P}^{x_1}(A_{k-1}).\] Therefore 
\begin{equation}
\begin{split}
\mathbb{E}^x(\tau-1)&=\sum_{k=1}^\infty (k-1)\mathbb{P}^x(A_k(x))\\
&= \sum_{k=2}^\infty \int_B p(x_0,x_1)\mathbb{P}^{x_1}(A_{k-1}(x_1))d\mu(x_1) \\
&= \int_B p(x,y)(\sum_{k=1}^\infty \mathbb{P}^y (A_k(y))) d\mu(y)\\
&= \int_B p(x,y)E_B(y)d\mu(y).\\
\end{split}
\end{equation} The result then follows.
\end{proof}

We will require a fairly general version of the Kolmogorov Extension Theorem. The following version found on p.523 in \cite{Hitch} suffices. Recall that a Polish space is a separable, completely metrizable topological space; which means that it possesses a countable basis, and there exists at least one metric that is complete and with metric topology equal to the original topology.

\begin{proposition}
Let $(W_t,\Sigma_t)_{t\in T}$ be a family of Polish spaces, and for each finite subset $F$ of $T,$ let $\mu_F$ be a probability measure on $\Omega_F=\prod_{t\in F} W_t$ with its product Borel $\sigma-$algebra $\Sigma_F.$ Assume the family $(\mu_F)$ satisfies the consistency condition that if $F'\subset F$ then $\mu_F|_{\Sigma_{F'}}=\mu_{F'}$.  Then there is a unique probability on the infinite product $\sigma-$algebra $\bigotimes_{t\in T}\Sigma_t$ that extends each $\mu_F.$
\end{proposition}

\begin{proposition} \label{kol}
Suppose $\lambda:X\rightarrow (0,\infty)$ is bounded and measurable. For $x_0\in X$ there exists a measure $P^{x_0}$ defined on $\Omega_{x_0}:=\{\omega:\mathbb{Z}_+\rightarrow [0, \infty)\times X\;|\;\omega(0)=(0,x_0)\}$ such that for cylinder sets of the form $A=\{\omega \in \Omega_{x_0}\;|\;\omega(j)\in A_j\times U_j, j=1,...,n\},$ we have that 
\begin{equation*}
\mathbb{P}^{x_0}(A)=\left(\prod_{j=1}^n \int_{A_j}\int_{U_j} \frac{e^{-t_{j}/ \lambda(x_{j-1})}}{\lambda(x_{j-1})}p(x_{j-1},x_{j}) \right ) dt_nd\mu(x_{n})...dt_1d\mu(x_{1}).
\end{equation*}
 
\end{proposition}
\begin{proof}
Let $T=\mathbb{Z}_+\setminus\{0\}.$ For $t\in T$ let $W_t=[0,\infty)\times X.$ Then since $[0,\infty)$ and $X$ are complete metric spaces, so is $[0,\infty)\times X.$ Let $\Sigma_t$ be the Borel $\sigma\mhyphen$ algebra on $W_t.$ For $\{j_1,...,j_n\}=F\subset T$, where $j_1<j_2<...<j_n,$ and $x=x_0\in X,$ define a probability measure $\mathbb{P}^x_F$ on $\Sigma_F$ by
\begin{equation*}
\begin{split}
& \mathbb{P}^x_F(\omega(j_i)\in A_i\times U_i,\;i=1,...,n)=\\ &\left(\prod_{i=1}^n \int_{A_i}\int_{U_i} \frac{e^{-t_{i}/ \lambda(x_{i-1})}}{\lambda(x_{i-1})}p(x_{i-1},x_{i}) \right ) dt_nd\mu(x_{n})...dt_1d\mu(x_{1}).
\end{split}
\end{equation*}

Since $X\times [0,\infty)$ is a Polish space, we need only check the consistency condition. Let $F,F' \subset \mathbb{Z}_+$ with $F\subset F'$. We may assume $F=\{1,...,n\}\subset F'=\{1,...,m\}$ and $n<m.$  Then since 
\[\int_X\int_{[0,\infty)}\frac{e^{-t/ \lambda(a})}{\lambda(a)}p(a,x)dtd\mu(x)=1,\] for any $a\in X,$ it is clear that  \[\mathbb{P}^x_{F'}|_{\Sigma_F}(\omega(j)\in A_j\times U_j,j\in F)= \mathbb{P}^x_F(\omega(j)\in A_j\times U_j,j\in F),\] so the consistency conditions hold. Hence by the Kolmogorov Extension Theorem there exists a measure on $\mathbb{P}^x$ on $\Omega_x$ extending the above definition.
\end{proof}

For $t\geq 0$, $x\in X,$ let $\hat{t}:\Omega_x\rightarrow \mathbb{Z}_+$ be defined by $\hat{t}(\omega):=\sup\{k\;|\;\sum_{j=0}^k \omega(j)_1\leq t\}.$ Then define $(X_t)_{t\in T}$ by $X_t(\omega)=\omega(\hat{t}(\omega))_2.$ Let $\tau_{\lambda,B}:=\inf\{t\geq 0\;|\;X_t\notin B\}.$ Then for $x\in X,$ let $E_{\lambda,B}(x):=\mathbb{E}^x \tau_{\lambda,B}.$ 

\begin{proposition} \label{eq1} For $x\in B$ we have $E_{\lambda,B}(x)=\lambda(x)+\int p(x,y)E_{\lambda,B}(y)d\mu(y).$ If $x\notin B$ we have $E_{\lambda,B}(x)=0.$
\end{proposition}

\begin{proof} The second claim is clear. If $x\notin B$ and $\omega\in \Omega_x,$ then $\hat{0}(\omega)=0$ and $\omega(0)_2=x.$ Hence $\tau_{\lambda,B}(\omega)=0.$ So $E_{\lambda,B}(x)=\mathbb{E}^x\tau_{\lambda,B} =0.$

Now suppose $x\in B.$ Let $\tau_1(\omega):=\inf\{t\geq 0\;|\;\hat{t}(\omega)=1\}=\omega(1)_1.$ So $\tau_1$ has exponential distribution with mean $\lambda(x).$ Then $\mathbb{E}^x \tau_{\lambda,B} = \mathbb{E}^x(\tau_1+(\tau_{\lambda,B}-\tau_1))=\lambda(x)+\mathbb{E}^x(\tau_{\lambda,B}-\tau_1).$ Hence we concentrate on the $\mathbb{E}^x(\tau_{\lambda,B}-\tau_1)$ term. For convenience of notation let $\tau=\tau_{\lambda,B}$.

\begin{equation*}
\begin{split}
\mathbb{P}^x(\tau-\tau_1\geq t)&=\mathbb{P}^x(X_{\tau_1+s}\in B \mbox{\;for\;} 0\leq s< t)\\
&= \int p(x,y)\mathbb{P}^y(X_s\in B \mbox{\;for\;} 0\leq s< t)d\mu(y) \\
&= \int p(x,y)\mathbb{P}^y(\tau\geq t)d\mu(y).\\
\end{split}
\end{equation*} Hence the distribution of $\tau-\tau_1$ is defined by $\mathbb{P}^x(\tau-\tau_1\in A)=\int p(x,y)\mathbb{P}^y(\tau\in A)d\mu(y).$ Therefore 
\[\mathbb{E}^x(\tau-\tau_1)=\int p(x,y)\mathbb{E}^y(\tau) d\mu(y).\] The result then follows.
\end{proof}

\begin{lemma} \label{contmeas} For $r>0$ the map $x\mapsto \mu(B_r(x))$ is continuous. 
\end{lemma}
\begin{proof} Let $x\in X$ and $\epsilon>0.$ By continuity of measure, there exists a $\delta>0$ so that $|\mu(B_{r+\delta}(x)-\mu(B_{r-\delta}(x))|<\epsilon.$ Let $y\in X$ with $d(x,y)<\delta.$ Then by the triangle inequality, $B_{r-\delta}(x)\subset B_r(x)\cap B_r(y)$ and $B_r(x)\cap B_r(y)\subset B_{r+\delta}(x).$ Hence $|\mu(B_r(x))-\mu(B_r(y))|\leq \mu(B_r(x)\triangle B_r(y))\leq \mu(B_{r+\delta}(x))-\mu(B_{r-\delta}(x))<\epsilon.$

\end{proof}

For $r>0,$ $x,y\in X$ let $p_r(x,y):=\frac{1}{\mu(B_r(x))}\chi_{B_r(x)}(y)$. Define a Markov kernel $P_r$ by $dP_r(x,\cdot)=p_r(x,\cdot)d\mu$ for $x\in X.$ We also denote the corresponding Markov operator by $P_r.$ That is for $f\in L^2(X,\mu),$ 
\[(P_r f)(x)=\frac{1}{\mu(B_r(x))}\int_{B_r(x)} f(y)d\mu(y).\] Let $(Y(r)_k)_{k=0}^\infty$ be the discrete time random walk generated by the Markov kernel $P_r$. In \cite{olliv} (with $r$ replaced by $\epsilon$) this is called the ``$\epsilon-$step random walk".

Let $R>0$ and $x_0\in X.$ Let $B:=B_R[x_0].$  Then let $\tau_{B,r}:=\inf\{k\;|\;Y(r)_k\notin B\}.$ Then let $E_{B,r}(x)=\mathbb{E}^x\tau_{B,r}.$ We will often omit the $B$ in the notation, writing $E_r(x)$ instead of $E_{B,r}(x).$

Then we have $E_r(x)=0$ for $x\in B^c$. For $x\in B,$ by Proposition \ref{markov} we have 
\[E_r(x)=1+\frac{1}{\mu(B_r(x))}\int_{B_r(x)}E_r(y)d\mu(y).\] 

\begin{exmp} \label{euclid}
Let $B$ the closed ball of radius $R$ about the origin in $\mathbb{R}^n$ under the Euclidean norm $|\cdot|$. Then let $X$ be a compact subset of $\mathbb{R}^n$ containing $B_{R+1}(x).$ Let $x\in B^\circ.$ Let $r$ be small enough so that $B_r(x)\subset B.$ Then using the process defined by uniform jumps in a ball of radius $r$ according to the Lebesgue measure, we have \[E_r(x)=\left(\frac{n+2}{n}\right)\frac{R^2-|x|^2}{r^2}.\]
\end{exmp}

\begin{lemma} Suppose $E_r$ is bounded above. Then $E_r$ is continuous on $B.$ 
\end{lemma}
\begin{proof} Let $x\in B$ and $\epsilon>0.$ We may assume that $\mu(B_r(x))>2\epsilon.$ Then we may choose a $\delta_0>0$ so that if $d(x,y)<\delta_0$ then $\mu(B_r(y))>\epsilon.$  Let $E^+=\sup_{z\in B}E_r(z).$ Let $\delta>0$ with $\delta<\delta_0$ so that if $y\in B$ and $d(x,y)<\delta$ then $\mu(B_r(x)\triangle B_r(y))<\frac{\epsilon}{2E^+}$ and $|\frac{1}{\mu(B_r(x))}-\frac{1}{\mu(B_r(y))}|<\frac{\epsilon}{2\mu(B_r(x))E^+}.$ Then 
\[|E_r(x)-E_r(y)|<\frac{\epsilon}{2}+\frac{1}{\mu(B_r(y))}\int_{B_r(x)\triangle B_r(y)}E_r(w)d\mu(w)<\epsilon.\]
\end{proof}

Let \[E^+_{r,B}:=\sup_{y\in B} E_{r,B}(y).\]
For $\beta>0$ let \[T_\beta(B):=\lim \sup_{r\rightarrow 0^+} E^+_{r,B} r^\beta.\] 
\begin{proposition} There exists a unique $\beta(B)\in [0,\infty]$ defined by \[\beta(B):=\sup\{\beta\;|T_\beta(B)=\infty\}=\inf\{\beta\;|\;T_\beta(B)=0\}\] with the property that if $\gamma<\beta(B)$ then $T_\gamma (B) = \infty,$ and if $\gamma>\beta(B)$ then $T_\gamma(B)=0.$ Moreover, if $B'=B_{R'}[x_0]$ with $R'<R$ then $\beta(B')\leq \beta(B).$
\end{proposition}
\begin{proof} 
If $0\leq \gamma<\beta$ then $E_{r,B}r^\beta = r^{\beta-\gamma} E_{r,B}r^{\gamma}.$ Hence if $T_\beta(B)>0$ then $T_\gamma(B)=\infty,$ and if $T_\gamma(B)<\infty$ then $T_\beta(B)=0.$ Monotonicity is clear from $E_{r,B'}\leq E_{r,B}.$
\end{proof}
Then let $\beta(x):=\inf_{R>0} \beta(B_R(x)).$ We call $\beta$ the \textit{local time exponent.} 

\begin{proposition} The local time exponent $\beta$ is upper semicontinuous. In particular $\beta$ is Borel measurable and bounded above. Moreover, we have $\beta\geq 1.$\label{semicont1}
\begin{proof} The proof of the first statements follows the same steps as Proposition \ref{semicont} and is omitted. To show 
$\beta\geq 1,$ suppose $x\in X$ and $R>0.$ Let $0<r<R.$ Let 
$B=B_R(x).$ Then we must have that $E_{B,r}(x)$ is greater than or equal to the smallest $k$ such that $x_0,x_1,...,x_k\in X$ 
with $x_0=x,$ $d(x_i,x_{i+1})\leq r$ for $i=0,...,k,$ $x_i\in B$ 
for $i=1,...,k-1,$ and $x_k\notin B$. Then $R\leq d(x_0,x_k)\leq 
\sum_{i=0}^{k-1}d(x_i,x_{i+1})\leq kr.$ Hence $E_{B,r}(x)\geq 
\frac{R}{r}.$ It then follows that $\beta(x)\geq 1.$ \end{proof}
\end{proposition}
\begin{exmp}
Consider a closed ball $B$ of radius $R$ about the origin in $\mathbb{R}^n$ as in Example \ref{euclid}. Then for $x\in B_R(0),$ $\beta(x)=2.$ This follows immediately from the formula in Example \ref{euclid}. Note that we also have that $T_2(B_R(x))\asymp R^2.$ 
\end{exmp}

\subsection{Time Scale Re-Normalization}
Recall that the Hausdorff measure with constant dimension is used to calculate the (constant) Hausdorff dimension of open balls which in turn are used to determine the local dimension. Then, in a process of re-normalization, the local dimension is used in a new measure that takes into account the finer local properties given by the local dimension. Guided by this analogy, we re-normalize the walks $Y(r)$ to take into account the local time scaling.

For $r>0$, $x\in X,$ let $\tau_r(x):=r^{\beta(x)}.$  Let $\Omega=(\mathbb{R}_+\times X)^{\mathbb{Z}_+\setminus\{0\}}.$ For each $x_0\in X, r>0$, applying Proposition \ref{kol} with $\lambda(x)=\tau_r(x)$  yields a probability 
measure $\mathbb{P}^{x_0}_r$ on $\Omega$ defined on ``cylinder sets" as follows.

For $A_1,A_2,...,A_n\subset \mathbb{R}_+$ measurable and $U_1,U_2,...,U_n\subset X$ measurable,
\begin{equation*} 
\begin{split}
&\mathbb{P}^{x_0}_r(\{\omega\in (\mathbb{R}^+\times X)^{\mathbb{Z}_+}\;|\;\omega(i)\in A_i\times B_i\;\mbox{for}\;i=1,...,n\})\\ &= \left(\prod_{i=1}^n \int_{A_i\times U_i}  \frac{e^{-t_i/\tau_r(x_{i-1})}}{\tau_r(x_{i-1})\mu(B_r(x_{i-1})}\chi_{B_r(x_{i-1})}(x_i)\right )dt_nd\mu(x_n)...dt_1d\mu(x_1).
\end{split}
\end{equation*}

Then, by Proposition \ref{kol}, there exists a probability measure on $\Omega$ extending the above definition.

Let $\Omega_{x_0}=\{\omega:\mathbb{Z}_+\rightarrow [0,\infty)\times X\;|\; \omega(0)=(0,x_0)\}.$ Then $\mathbb{P}^{x_0}_r$ may be considered as a measure on $\Omega_{x_0}.$ For $t\geq 0,$ define $\hat{t}$ by \[\hat{t}(\omega)=\sup\{k\geq 0\;|\;\sum_{j=1}^{k} \omega(j)_1 \leq t\}.\] Then set $X^{(r)}_t(\omega):=\omega(\hat{t})_2.$ 

\begin{proposition} \label{generator} $(X^{(r)}_t)_{t\geq 0}$ is a continuous time Markov process with 
generator \[-\frac{d}{dt}\Bigm |_{t=0} \mathbb{E}^{x} f(X^{(r)}_t) = \mathscr{L}_rf(x):=\frac{1}{r^{\beta(x)}\mu(B_r(x))}\int_{B_r(x)}(f(x)-f(y))d\mu(y).\]
\end{proposition}
\begin{proof} It's clear from the definition that $(X^{(r)}_t)_{t\geq 0}$ satisfies the Markov property. 

Now since, by Proposition \ref{semicont1}, there exists an $M>1$ 
such that $1\leq \beta \leq M$, we have that there exists a $C_1>0$ 
such that $C_1\leq r^{\beta(y)}$ for all $y\in X.$ Since the map 
$y\mapsto \mu(B_r(y))$ is continuous by Lemma \ref{contmeas}, 
there exists $C_2, C_3>0$ such that $C_2\leq \mu(B_r(y))\leq C_3$ for all 
$y\in X.$

Let $f\in L^2(X,\mu),$ $t>0.$ Set $x_0=x.$ Since $X$ is compact, $\|f\|_{1,\mu}<\infty.$ Note that $\mathbb{E}^{x} f(X^{(r)}_t)$ consists of an infinite sum from $k=0$ to $\infty$ of terms of the form \[\left( \prod_{i=1}^k \int_0^{t-\sum_{j=1}^{i-1} t_j} \int_X \frac{e^{\frac{-t_{i}}{\tau_r(x_{i-1})}}p_r(x_{i-1},x_i)}{\tau_r(x_{i-1})}\right)\int_{t-\sum_{j=1}^k t_j}^\infty  \frac{e^{\frac{-t_{k+1}}{\tau(x_k)}}f(x_k)}{\tau(x_k)}\left(\prod_{j=0}^{k-1}dt_{k+1-j}d\mu(x_{k-j})\right)dt_1.\]
However, for $k\geq 2$ the absolute value of the $k\mhyphen$th term is bounded above by $\|f\|_{1,\mu}(Ct)^k,$ for some $C>0$ depending only on $C_1$, $C_2$, $C_3$, which in turn is bounded above, for $0< t<C,$ by the $k\mhyphen$th term of an absolutely convergent series. Hence, dividing by $t,$ one sees that only the $k=0$ and $k=1$ terms contribute to the derivative evaluated at $t=0.$

Let us then consider the $k=0$ and $k=1$ terms separately. The term for $k=0$ is, recalling the convention that the empty product is $1$ and the empty sum is $0,$ given by
\[\int_t^\infty \frac{e^{-t_1/\tau_r(x)}f(x)}{\tau_r(x)}dt_1=e^{-t/\tau_r(x)}f(x).\] Hence the derivative of the term for $k=0$ evaluated at $t=0$ is  $\frac{-f(x)}{\tau_r(x)}.$

Now the term for $k=1$ is \begin{equation*}
\begin{split}
&\int_0^{t}\int_X\int_{t-t_1}^\infty \frac{e^{-t_1/\tau_r(x)}p_r(x,x_1)e^{-t_2/\tau_r(x_1)}f(x_1)}{\tau_r(x)\tau_r(x_1)}dt_2d\mu(x_1)dt_1 \\
&=\int_X\frac{p_r(x,x_1)}{\tau_r(x)}\int_0^{t} e^{-t_1/\tau_r(x)}e^{-(t-t_1)/\tau_r(x_1)}f(x_1)dt_1d\mu(x_1).
\end{split}
\end{equation*}
However, \[\frac{e^{-t/\tau_r(x_1)}}{t}\int_0^t e^{-t_1(\tau_r(x)^{-1}-\tau_r(x_1)^{-1})}dt_1\; \overrightarrow{_{t\rightarrow 0^+}}\; 1\] for all $x_1\in X.$ 
Hence, since $f$ is $\mu\mhyphen$integrable and \[\left|
\frac{p_r(x,x_1)}{\tau_r(x)t}f(x_1)\int_0^{t} e^{-
t_1/\tau_r(x)}e^{-(t-t_1)/\tau_r(x_1)}dt_1\right|\leq \frac{|f(x_1)|}
{C_1C_2},\] by the Dominated Convergence Theorem we have that 
\[\lim_{t\rightarrow 0^+} \frac{1}{t} \int_X\frac{p_r(x,x_1)}
{\tau_r(x)}\int_0^{t} e^{-t_1/\tau_r(x)}e^{-(t-
t_1)/\tau_r(x_1)}f(x_1)dt_1d\mu(x_1) = \frac{1}{\tau_r(x)}\int_X 
p_r(x,x_1)f(x_1)d\mu(x_1).\]
Therefore \[\mathscr{L}_r f(x)=\frac{1}{\tau_r (x)}\left(f(x)-\int_X p_r(x,x_1)d\mu(x_1)\right) =\frac{1}{r^{\beta(x)}\mu(B_r(x))}\int_{B_r(x)}(f(x)-f(y))d\mu(y),\] as desired.

\end{proof}
Observe that for the process $(X^{(r)}_t)_{t\geq 0}$, there is an equilibrium probability measure $\nu_r$ with density \[d\nu_r(x)/d\mu(x):=\frac{r^{\beta(x)}\mu(B_r(x))}{Z_r},\] where $Z_r$ is the normalization factor defined by $Z_r=\int_X r^{\beta(z)}\mu(B_r(z))d\mu(z).$ 

Let \[\tau_{r,B}:=\inf\{t\;|\;X^{(r)}_t\notin B\}.\]
We now define \textit{exit time functions}. 
For $x\in X,$ let \[\phi_{r,B}(x):=\mathbb{E}^x\tau_{r,B}.\]

\begin{proposition} \label{equation} Let $B$ a non-empty ball and $r>0.$ We have $\mathscr{L}_r\phi_{r,B}(x)=1$ for $x\in B,$ and $\phi_r(x)=0$ for $x\notin B.$ 
\end{proposition}
\begin{proof} By Proposition \ref{eq1} we have that $r^{\beta(x)}=(I-P_r)\phi_{r,B}(x)$ for $x\in B.$ Hence, by Proposition \ref{generator}, $\mathscr{L}_r\phi_{r,B}(x)=1$ for $x\in B.$ It follows immediately from the definition of $\tau_{r,B}$ that $\phi_{r,B}$ vanishes outside of $B.$
\end{proof}

Let $\phi^+_{r,B}:=\sup_{y\in B} \phi_{r,B}(y).$
For $\beta>0$ let $\mathscr{T}(B):=\lim \sup_{r\rightarrow 0^+} \phi^+_{r,B}.$

\begin{definition} For $\beta(x)$ the local time exponent, we say $(X,d,\mu)$ satisfies $(E_\beta)$, or that $(X,d,\mu)$ satisfies the \textit{variable time regularity condition} with exponent $\beta$, if for all $x\in X, 0<r<\frac{\diam(X)}{2},$ we have 
\[\mathscr{T}(B_r(x))\asymp r^{\beta(x)}.\]
\end{definition}

\begin{exmp}
Let $B$ the open ball of radius $R$ about the origin in $\mathbb{R}^n$ under the euclidean norm $|\cdot|$. Then let $X$ be a compact subset of $\mathbb{R}^n$ containing $B_{R+1}(x).$ Let $x\in B.$ Let $r$ small enough so that $B_r(x)\subset B.$ Then we have \[\phi_r(x)=\left(\frac{n+2}{n}\right)(R^2-|x|^2).\] Hence $\mathscr{T}(B_r(x)) = (\frac{n+2}{n}) r^2,$ so that $(E_2)$ is satisfied.
\end{exmp}

Recall that a function $\varphi$ on a metric space $(X,\rho)$ is log-H{\"o}lder continuous if 
there exists a $C>0$ such that $|\varphi(x)-\varphi(y)|\leq \frac{-C}{\log(\rho(x,y))}$ for all $x,y$ with $0<\rho(x,y)<\frac{1}{2}.$

The following is an analog of Lemma \ref{log} for $\beta$ instead of $\alpha$.
\begin{lemma}
If $X$ is variable time regular with exponent $\beta$ then $\beta$ is log-H{\"o}lder continuous.
\end{lemma}
\begin{proof}
The proof follows the same steps as Lemma \ref{log} and is omitted.
\end{proof}
It follows that the map $x\mapsto r^{\beta(x)}$ is continuous in the case that $X$ is variable time regular with exponent $\beta.$ We then have the following proposition.

\begin{proposition}
Let $r>0$ and $B$ a non-empty ball. Suppose $X$ is variable time regular with exponent $\beta$. Then $\phi_{r,B}$ is continuous on  $B.$
\end{proposition}
\begin{proof} The variable time regularity condition implies $\phi_{r,B}$ is bounded. Since the map $x\mapsto r^{\beta(x)}\mu(B_r(x))$ is continuous and $\mathscr{L}_r\phi_{r,B}(x)=1$ for $x\in B,$ continuity on $B$ follows from a slight modification of the proof of Lemma 6.6. We omit the details.
\end{proof}

\section{Green's functions and Dirichlet spectrum}
Let $R>0, x_0\in X,$ $B=B_R(x_0).$ We may assume $R>0$ is small enough that $B_r[x_0]^c$ is non-empty.  For $r>0$ consider the random walk killed on exiting $B.$ Let $P_r^B$ be the Markov operator with kernel defined by the function
\[p^B_r(x,y)= \chi_{B\times B}(x,y) p_r(x,y)\] where $p_r(x,y)=\frac{1}{\mu(B_r(x))}\chi_{B_r(x)}(y)$ is the function generating the Markov kernel $P_r$ for $Y(r).$
Let $v(r,x)=\mu(B_r(x)).$ Then define $\mu_r$ by 
\[d\mu_r(x)=v(r,x)d\mu(x).\]

Note $\mu_r$ is $P_r$ invariant. For $f$ defined on $B$ we let $f\chi_B$ the function defined on all of $X$ by extending it to equal $0$ outside of $B.$ Hence if $f\in L^1(B,\mu_r),$ then 
\begin{align*}
\|P_r^Bf(x)\|_{L^1(B,\mu_r)}&=\int_B |P_r^Bf(x)|d\mu_r(x) \leq \int_B |P^B_r|f|(x)|d\mu_r(x) \\
&= \int \int \chi_B(x)\chi_B(y)p_r(x,y)v(r,x)|f(y)|d\mu(y)d\mu(x)\\& = 
\int \int \chi_B(x)\chi_B(y)v(r,y)p_r(y,x)|f(y)|d\mu(x)d\mu(y)\\
&\leq \int \chi_B(y)|f(y)|v(r,y)d\mu(y)= \|f\|_{L^1(B,\mu_r)}.
\end{align*}
Therefore $\|P^B_r\|_{\mathscr{B}(L^1(B,\mu_r))}\leq 1.$ In particular its spectral radius, as an operator on $L^1(B,\mu_r),$ is at most one. 

Note also that we may consider $P^B_r$ as an operator on $L^2(B,\mu_r).$ In this capacity if $f,g\in L^2(B,\mu_r)$ then we have 
\begin{align*}
\langle f, P^B_rg \rangle_{L^2(B,\mu_r)} &= \int_B f(x)P^B_rg(x)v(r,x)d\mu(x) \\
&= \int\int \chi_B(x)\chi_B(y)v(r,x)p_r(x,y)g(y)f(x)d\mu(y)d\mu(x) \\
&= \int\int \chi_B(x)\chi_B(y)v(r,y)p_r(y,x)g(y)f(x)d\mu(x)d\mu(y) \\
&= \int_B g(y)P^B_rf(y)v(r,y)d\mu(y)= \langle P^B_rf, g \rangle_{L^2(B,\mu_r)}.
\end{align*}
Hence $P^B_r$ is self adjoint. Note that since $P^B_r$ has a square integrable kernel it is compact.

 We also define operators $L^B_r$ on $L^2(B,\mu_r)$ and $\mathscr{L}^B_r$ on $L^2(B,\nu_r)$ as follows. For $f\in L^2(B,\mu_r),$ $L^B_rf(x)=f(x)-P^B_rf(x)$, and for $f\in L^2(B,\nu_r),$ $\mathscr{L}^B_rf(x)=\frac{1}{r^{\beta(x)}}(f(x)-P^B_rf(x)).$

For convenience we state without proof the following classic result about the spectrum of a compact operator on a Banach space.  For more information and proofs see, for instance, \cite{Simon}.

\begin{proposition} (Riesz-Schauder) 
If $A$ is compact on a Banach space then $\sigma(A)\setminus\{0\}$ is discrete and contains only eigenvalues of finite multiplicity. The eigenvalues may only accumulate at $0.$ 
\end{proposition}

Hence the non-zero spectrum of $P^B_r$ consists entirely of (real) eigenvalues. Suppose $\lambda$ is an $L^2(B,\mu_r)$ eigenvalue. Then if $f\in L^2(B,\mu_r)$ is a corresponding eigenvector, since $\mu_r$ is a finite measure, $f\in L^1(B,\mu_r)$ . So $\lambda$ is also in the spectrum of $P^B_r$ as an operator on $L^1(B,\mu_r)$. Hence $|\lambda|\leq 1.$ It follows that the $L^2(B,\mu_r)$ spectrum of $P^B_r$ is contained in $[-1,1].$ Moreover, since $P^B_r$ is normal its spectral radius equals its operator norm (see Theorem 2.2.11 in \cite{Simon}).

We now show that neither $1$ nor $-1$ are in the spectrum. 

\begin{lemma} $\|P^B_r\|_{\mathscr{B}(L^2(B,\mu_r))}<1.$
\end{lemma}
\begin{proof}
We already know that $\sigma(P^B_r)\subset[-1,1]$. If $1,-1$ are in the spectrum then they are eigenvalues since $P^B_r$ is compact. 
Suppose $(I-P^B_r)f=0$ for some $f\in L^2(B,\mu_r).$ Then 
\[\langle f, (I-P^B_r)f\rangle_{L^2(B,\mu_r)}=\int_X \int_{B_r(x)} |\chi_{B}(y)f(y)-\chi_{B}(x)f(x)|^2d\mu(y)d\mu(x)=0.\]

Hence, since $X$ is compact there exist finitely many balls $B_r(x_1),...,B_r(x_n)$ covering $X$ with $f$ $\mu\mhyphen$a.e. constant on each $B_r(x_i).$ Since $X$ is connected, the graph formed with vertices $x_1,...,x_n$ and edge relation $x_i\sim x_j$ if and 
only if $B_r(x_i)\cap B_r(x_j)$ is connected. It follows that $\chi_Bf$ is $\mu\mhyphen$a.e constant, where we may extend $\chi_Bf$ outside of $B$ by setting it equal to $0.$ Then since $B_R[x_0]$ is closed and we have assumed its complement is non-empty, its complement contains an open set. Since $\mu$ has full support, $\mu(B^c)>0.$ But $f\chi_B$ is identically $0$ on $B^c.$ It follows that $f\chi_B$ is $0$ $\mu\mhyphen$a.e. Note that $v(r,\cdot)$ is bounded below since it is continuous on the compact space $X.$ Hence $f\chi_B$ is $0$ $\mu_r\mhyphen$a.e. In particular $f=0 \in L^2(B,\mu_r).$

Now suppose $(I+P^B_r)f=0$ for some $f\in L^2(B,\mu_r).$ Then 
\[\langle f, (I+P^B_r)f\rangle_{L^2(B,\mu_r)}=\int_X \int_{B_r(x)} |\chi_{B}(y)f(y)+\chi_{B}(x)f(x)|^2d\mu(y)d\mu(x)=0.\]  
Hence there exist $x_1,...,x_n$ in $X$ with $(B_r(x_j))_{j=1}^n$ an open cover of $X$ 
such that $\chi_Bf(y)=-\chi_Bf(x_i)$ for $\mu\mhyphen$a.e. $y\in B_r(x_i)$ for each $i=1,..,n.$ In particular $f\chi_B$ is $\mu\mhyphen$a.e constant on each $B_r(x_i)$. Suppose $i\neq j$. 
Then since the graph formed with vertices $x_1,...,x_n$ by the non-empty intersection relation is connected, it follows that there is a path in the graph connecting $x_i$ and $x_j$. However, if $x_k\sim x_i$ then $B_r(x_i)\cap B_r(x_k)$ is non-empty and open. Moreover, for $\mu\mhyphen a.e.$ $y$ in the intersection, $\chi_b(y)f(y)=-\chi_B(x_i)f(x_i)=-\chi_B(x_k)f(x_k).$ Hence 
$\chi_B(x_i)f(x_i)=\chi_B(x_k)f(x_k).$ Continuing in this way along the path leads to $\chi_B(x_i)f(x_i)=\chi_B(x_j)f(x_j).$ Therefore $\chi_Bf$ is constant on $x_1,...,x_n.$ It follows by the connectivity of the induced graph that $\chi_Bf$ is 
$\mu\mhyphen a.e.$ constant. This implies that it is $\mu_r\mhyphen$a.e. constant. Since $B^c$ has non-zero measure, $f$ is $\mu_r\mhyphen$a.e. equal to $0.$

Therefore, since the spectrum away from $0$ is discrete, there exists an $\epsilon>0$ so that $\sigma(P^B_r)\subset [-1+\epsilon,1-\epsilon].$ In particular $\|P^B_r\|<1.$
\end{proof}
Now consider the following Dirichlet problem on $B,$
\[\mathscr{L}_ru(x)=f(x),\;\;\mbox{for}\;x\in B,\;u,f\in L^2(X,\mu),\;u(x)=0\;\;\mbox{for\;} x\in B^c.\]
We show that we can find a Green function $G^B_r\in L^1(X\times X, \;\mu\otimes \mu)$ solving the Dirichlet problem with 
\[u(x)=\int G^B_r(x,y)f(y)d\mu(y).\] Recall $d\nu_r(x)=\frac{r^{\beta(x)}\mu(B_r(x))}{Z_r},$ where $Z_r$ is the normalization factor $\int r^{\beta(x)}\mu(B_r(x))d\mu(x);$ and $\mathscr{L}^B_r:L^2(B,\nu_r)\rightarrow L^2(B,\nu_r)$ is defined by 
\[\mathscr{L}^B_rf(x)=\frac{1}{r^{\beta(x)}}(I-P^B_r)f(x).\]

\begin{theorem} We have that $(L^B_r)^{-1}$ is a bounded operator on $L^2(B,\mu)$ which has an integral kernel $K^B_r\in L^1(B\times B,\;\mu \otimes \mu)$. Moreover, $(L^B_r)^{-1}$ is a bounded positive operator on $L^2(B,\mu_r)$, and the function \[k_r^B(x,y):=\frac{K^B_r(x,y)}{\mu(B_r(y))}\] is a symmetric integral kernel for $(L^B_r)^{-1}$ in $L^2(B,\mu_r).$ 

Suppose $(E_\beta)$ holds. Then $(\mathscr{L}^B_r)^{-1}$ is a bounded operator on $L^2(B,\mu)$ which has an integral kernel $G^B_r\in L^1(B\times B,\;\mu \otimes \mu)$. Moreover, $(\mathscr{L}^B_r)^{-1}$ is a bounded positive operator on $L^2(B,\nu_r)$, and the function \[g_r^B(x,y):=\frac{G^B_r(x,y)}{\mu(B_r(y))r^{\beta(y)}}\] is a symmetric integral kernel for $(\mathscr{L}^B_r)^{-1}$ in $L^2(B,\nu_r).$ 
\end{theorem}

\begin{proof} First note that since $r$ is fixed and the maps $x\mapsto r^{\beta(x)}$ and $x\mapsto \mu(B_r(x))$ are positive and continuous, we have $\mu\asymp \mu_r \asymp \nu_r.$ So $L^2$ convergence in any of the three measures implies convergence in all of them.

Lemma 7.2 implies that the Neumann series $\sum_{k=0}^\infty (P^B_r)^k$ converges in the operator norm topology to the bounded operator $(I-P^B_r)^{-1}.$
For each $j,$ let $K^B_{r,j}(\cdot,\cdot)$ be the integral kernel for $(P^B_r)^k.$ 

Set \[K^B_r(x,y):=\sum_{j=0}^\infty k^B_{r,j}(x,y).\]

Let $f\in L^2(B,\mu_r).$ Then let $h=(I-P^B_r)^{-1}f\in L^2(B,\mu_r).$ Let $h_n=\sum_{k=0}^n (P^B_r)^kf.$ Then by the convergence of the Neumann series, 
\[h_n\convn h\;\;\mbox{in\;} L^2(B,\mu_r).\] Note also that for each $x\in B$ and $n\in \mathbb{Z}_+$, 
\[\sum_{j=0}^n |\int K^B_{r,j}(x,y)f(y)d\mu(y)|\leq \sum_{j=0}^n (P^B_{r})^j|f|(x)\leq (I-P^B)^{-1}|f|(x).\] It follows for $\mu_r$ almost all $x\in B$ that 
\[\sum_{j=0}^\infty \int K^B_{r,j}(x,y)f(y)d\mu(y)\;\;\mbox{converges absolutely.}\]
Let $q$ denote the pointwise limit of this series. Then since $h_n\convn h$ in $L^2(B,\mu_r),$ there exists a subsequence $h_{n_j}$ such that \[h_{n_j} \convk h \;\;\mbox{pointwise}\;\mu_r\mhyphen\mbox{a.e.}.\] It follows that $h=q.$ But by monotone convergence, \[K^B_r(x,\cdot)|f(\cdot)|\in L^1(B,\mu_r)\;\;\mbox{for}\;\mu_r\mhyphen\mbox{a.e\;} x\in B.\]
Since this function dominates $|\sum_{j=0}^n K^B_{r,j}(x,\cdot)f(\cdot)|$ for each $n,$ by the Dominated Convergence Theorem it follows that 
\[(I-P^B_{r})^{-1}f(x)=\int K^B_r(x,y)f(y)d\mu(y).\]

Note further that $K^B_r\in L^1(\mu \otimes\mu)$ since $\mu(B_r(x))$ is bounded below by a positive constant independent of $x,$ $\int K^B_r d(\mu\otimes \mu)=\|(I-P^B_r)^{-1}\chi_B \|
_{L^1(X,\mu)}=\|(I-P^B_r)^{-1}\chi_B\|_{L^1(B,\mu)}\leq C\|(I-P^B_r)^{-1}\chi_B\|
_{L^2(B,\mu_r)}<\infty$ for some constant $C>0.$

Note $L^B_r$ is self-adjoint on $L^2(B,\mu_r).$ Hence $(L^B_r)^{-1}$ is also self-adjoint. Therefore the integral kernel with respect to $\mu_r$ is symmetric. 
However, $(L^B_r)^{-1}f(x)=\int \frac{K_r^B(x,y)}{\mu(B_r(y))}f(y)d\mu_r(y).$

Therefore let $k_r^B(x,y)=\frac{K_r^B(x,y)}{\mu(B_r(y))}.$ Then $k^B_r$ is symmetric. Moreover, since 
\[\langle f, L^B_r f\rangle_{L^2(B,\mu_r)} = \int_X \int_{B_r(x)}|\chi_B(y)f(y)-\chi_B(x)f(x)|^2d\mu(y)d\mu(x)\geq 0,\]  $\mathscr{L}^B_r$ is a positive operator.

Set \[G^B_r(x,y)=K^B_r(x,y)r^{\beta(y)}.\]

We set $G^B_r(x,y)=0$ for $x\notin B$ or $y\notin B.$ Then if $\mathscr{L}_ru(x)=f(x)$ where $f=0$ outside of $B,$ then $\mathscr{L}^B_ru(x)=\chi_B(x)f(x).$

Hence \[(I-P^B_r)^{-1}(\chi_Bf)(x)=r^{\beta(x)}u(x).\]
Therefore, by our previous considerations, 

\[\chi_B(x) f(x) = \int K^B_r(x,y)r^{\beta(y)}u(y)d\mu(y)=\int G^B_r(x,y)u(y)d\mu(y).\] Since $K^B_r\in L^1(B\times B,\mu\otimes \mu),$ we also have $G^B_r\in L^1(B\times B,\mu\otimes \mu).$

Note $\mathscr{L}^B_r$ is self-adjoint on $L^2(B,\nu_r).$ Hence $(\mathscr{L}^B_r)^{-1}$ is also self-adjoint. Therefore the integral kernel with respect to $\nu_r$ is symmetric. 
However, $(\mathscr{L}^B_r)^{-1}f(x)=\int \frac{G_r^B(x,y)}{\mu(B_r(y))r^{\beta(y)}}f(y)d\nu_r(y).$

Therefore let $g_r^B(x,y)=\frac{G_r^B(x,y)}{\mu(B_r(y))r^{\beta(y)}}.$ Then $g^B_r$ is symmetric. Moreover, since 
\[\langle f, \mathscr{L}^B_r f\rangle_{L^2(B,\nu_r)} = \frac{1}{Z_r}\int_X \int_{B_r(x)}|\chi_B(y)f(y)-\chi_B(x)f(x)|^2d\mu(y)d\mu(x)\geq 0,\] $\mathscr{L}^B_r$ is a positive operator. Since $(I-P^B_r)^{-1}$ is bounded and $x\mapsto r^{\beta(x)}$ is bounded below, it follows that $(\mathscr{L}^B_r)^{-1}$ is also bounded. 
\end{proof}

\begin{corollary} $\phi_{B,r}(x)=\int_B g^B_r(x,y)d\mu(y)=\int_B g^B_r(x,y)d\nu_r(y).$
\end{corollary}
\begin{proof} This follows from the exit time equation $\mathscr{L}^B_r\phi_{B,r}=\chi_B.$
\end{proof}
The following proof is inspired by the elegant Lemma 2.3 in Telcs \cite{art}.
\begin{theorem} \label{main1} If $\sigma^B_r$ is the bottom of the spectrum of $L^B_r$ on $L^2(B,\mu_r)$ then \[\sigma^B_r\geq \frac{c}{E^+_{r,B}}\] for some constant $c$ independent of $r,$ $R.$ and $x_0.$ 

Suppose $(E_\beta)$ holds. Then if $\lambda^B_r$ is the bottom of the spectrum for $\mathscr{L}^B_r$ on $L^2(B,\nu_r)$ then  \[\lambda^B_r\geq \frac{c}{R^{\beta(x_0)}}\] for some constant $c>0$ independent of $r,R,$ and $x_0.$
\end{theorem}
\begin{proof} 
We prove the second claim. The first follows from a similar argument using the symmetric Green function $k^B_r$ for $L^B_r$ since $\int k^B_r(x,y)d\mu_r(y)=E_{r,B}(x).$

If $\frac{1}{\lambda^B_r}$ is in the spectrum of $(\mathscr{L}^B_r)^{-1}$ we may choose a sequence $(f_n)_{n=1}^\infty$ with each $f_n$ non-zero and in $L^2(B,\nu_r)$ with 
\[\|((\mathscr{L}^B_r)^{-1}-\frac{1}{\lambda^B_r})f_n\|_{L^2(B,\nu_r)} \convn 0.\]
Hence also \[\|((\mathscr{L}^B_r)^{-1}-\frac{1}{\lambda^B_r})f_n\|_{L^1(B,\nu_r)} \convn 0.\]
We may assume each $\|f_n\|_{L^1(B,\nu_r)}=1.$ Let $\epsilon>0.$ Then let $f_n$ with 
\[\frac{1}{\lambda^B_r}=\|\frac{1}{\lambda^B_r}f_n\|_{L^1(B,\nu_r)}\leq \|(\mathscr{L}^B_r)^{-1}f_n\|+\epsilon.\]
Then 
\begin{align*}
\|(\mathscr{L}^B_r)^{-1}f_n\|_{L^1(B,\nu_r)} &\leq \int_B\int_B g^B_r(x,y)|f_n(y)|d\nu_r(y)d\nu_r(x) \\ &= \int_B\int_B g^B_r(y,x)|f_n(y)|d\nu_r(x)d\nu_r(y) \\&= \int_B\phi_{B,r}(y)|f_n(y)|d\nu_r(y) \leq \phi_{B,r}^+ \|f_n\|_{L^1(B,\nu_r)}\\ &= \phi_{B,r}^+ \leq CR^{\beta(x_0)}.
\end{align*}

Hence $(\lambda^B_r)^{-1} \leq CR^{\beta(x_0)}+\epsilon.$ Since $\epsilon>0$ was arbitrary we have \[\frac{1}{\lambda^B_r} \leq CR^{\beta(x_0)}.\]
\end{proof}
\begin{corollary}
For all $f\in L^2(B,\nu_r)$ we have 
\[\langle f, \mathscr{L}^B_rf\rangle_{L^2(B,\nu_r)}\geq cR^{-\beta(x_0)}\|f\|_{L^2(B,\nu_r)}\] for some constant $c>0$ independent of $r,R,$ and $x_0.$
\end{corollary}
\begin{proof}
By the min-max principle $\lambda^B_r=\inf_{\|f\|_{L^2(B,\nu_r)}=1}\langle  f,\mathscr{L}^B_rf\rangle.$
\end{proof}
Part of the proof of the following lemma essentially follows the idea of the proof of Lemma 2.2 in Telcs \cite{telcs1}. See also the proof of Theorem 4.8 in \cite{grigorlau}.
\begin{theorem} Suppose $\mu$ is variable Ahlfors $\alpha$-regular. Then $\beta(x)\geq 2$ for all $x\in X.$ 
\end{theorem}

\begin{proof}  Let $x\in X,$ $R>0,$ and $0<r<1.$ We may assume $R$ is small enough so that $\mu(B_R[x]^c)>0.$ Let $\psi_{R,r,x}(y):=(\frac{R-d(x,y)}{r})\chi_{B_R(x)}(y).$ Then if $d(y,z)<r,$ \[|\psi_{R,r,x}(y)-\psi_{R,r,x}(z)|^2\leq 1.\] Since $\mu$ is Ahlfors $\alpha$-regular, by Lemma \ref{log}, $\alpha$ is continuous. Let $\beta^+(B_R(x))=\sup_{y\in B_R(x)}\alpha(y)$ and $\beta^-(B_R(x))=\inf_{y\in B_R(x)}\alpha(y).$ By the min-max principle and variable Ahlfors regularity, 
\[\sigma^{B_R(x)}_r\leq \frac{\langle \psi_{R,r,x}, L^B_r \psi_{R,r,x}\rangle_{L^2(B_R(x),\mu_r)}}{\|\psi_{R,r,x}\|_{L^2(B_R(x),\mu_r)}}\leq \frac{C_1\int_{B_R(x)}r^{\alpha(y)}d\mu(y)}{\int_{B_{R/2}(x)}|\psi_{R,r,x}(y)|^2r^{\alpha(y)}d\mu(y)},\] for some $C_1>0$ independent of $r, R,$ and $x.$ However, $|\psi_{R,r,x}|^2$ is bounded below by $R^2/4r^2$ on $B_{R/2}(x).$ Hence \[\sigma_r^{B_R(x)}\leq \frac{4r^2C_1\mu(B_R(x))}{r^{\alpha^+(B_R(x)-\alpha^-(B_R(x))R^2\mu(B_{R/2}(x))}}.\] However, it is not difficult to see that since $\mu$ is variable Ahlfors regular it satisfies the following doubling property: there exists a $C_2>0$ such that for all $\rho>0, y\in X,$ 
\[0<\mu(B_{2\rho}(y))\leq C_2\mu(B_\rho(y))<\infty.\] Let $C=4C_1C_2$. Then by Theorem \ref{main1}, there exists a $c>0$ such that $\sigma^{B_R(x)}_r\geq \frac{c}{E^+_{r,B_{R}(x)}}$ for all $r,R>0, x\in X.$ Therefore, putting these inequalities together we have,
\[cr^{\alpha^+(B_R(x))-\alpha^-(B_R(x))-2}\leq \frac{C}{R^2}E^+_{r,B_R(x)}.\] It then follows that 
\[\beta(B_R(x))\geq 2-(\alpha^+(B_R(x))-\alpha^-(B_R(x))),\] since otherwise there would exist a $\gamma$ with $\beta(B_R(x))<\gamma<2-(\alpha^+(B_R(x))-\alpha^-(B_R(x))).$ Then $\eta:=2-(\alpha^+(B_R(x))-\alpha^-(B_R(x)))-\gamma>0$ and \[\frac{c}{r^\eta}\leq \frac{C}{R^2}E^+_{r,B_R(x)}r^\gamma.\] Taking a $\limsup$ of both sides as $r\rightarrow 0^+$ then would imply a contradiction since the right hand side would be $0,$ by the definition of $\beta(B_R(x))$ as a critical exponent, and the left $\infty,$ since $\eta>0.$ Hence $\beta(B_R(x))\geq 2-(\alpha^+(B_R(x))-\alpha^-(B_R(x))),$ as claimed. However, $\alpha$ is continuous since $\mu$ is variable Ahlfors $\alpha$-regular. So letting $R\rightarrow 0^+$ yields \[\beta(x)\geq 2.\]
\end{proof}

\begin{section}{Acknowledgements} I would like to thank Dr. Jean Bellissard, Dr. Evans Harrell, Dr. Alexander Teplyaev, and Dr. Luke Rogers for helpful comments and discussion.
\end{section}

\bibliographystyle{amsplain}
\bibliography{bib}

\end{document}